\documentclass[11pt]{article}

\usepackage{etex}
\usepackage[margin={1.1in}]{geometry}
\usepackage[titletoc]{appendix}

\usepackage[small,nohug,heads=vee]{diagrams}
\diagramstyle[labelstyle=\scriptstyle]

\usepackage{tikz}
\usetikzlibrary{matrix}
\usepackage{eucal}
\usepackage{mathrsfs}
\usepackage{stmaryrd}

\usepackage{rotating}

\usepackage{longtable}

\usepackage[sc]{mathpazo}
\usepackage{courier}

\usepackage[utf8]{inputenc}
\usepackage[T1]{fontenc}

\usepackage[english]{babel}
\usepackage{blindtext}

\usepackage{amsmath}

\usepackage{microtype}

\usepackage{amsmath}
\usepackage{amssymb}
\usepackage{amsthm}
\usepackage{color}
\usepackage{latexsym,extarrows}

\usepackage[all]{xy}

\usepackage[stable,multiple]{footmisc}

\RequirePackage[font=small,format=plain,labelfont=bf,textfont=it]{caption}
\makeatletter
\newsavebox{\@brx}
\newcommand{\llangle}[1][]{\savebox{\@brx}{\(\m@th{#1\langle}\)}%
  \mathopen{\copy\@brx\kern-0.5\wd\@brx\usebox{\@brx}}}
\newcommand{\rrangle}[1][]{\savebox{\@brx}{\(\m@th{#1\rangle}\)}%
  \mathclose{\copy\@brx\kern-0.5\wd\@brx\usebox{\@brx}}}
\makeatother

\begin{document}
\def\e#1\e{\begin{equation}#1\end{equation}}
\def\ea#1\ea{\begin{align}#1\end{align}}
\def\eq#1{{\rm(\ref{#1})}}
\theoremstyle{plain}
\newtheorem{thm}{Theorem}[section]
\newtheorem{lem}[thm]{Lemma}
\newtheorem{prop}[thm]{Proposition}
\newtheorem{cor}[thm]{Corollary}
\theoremstyle{definition}
\newtheorem{dfn}[thm]{Definition}
\newtheorem{ex}[thm]{Example}
\newtheorem{rem}[thm]{Remark}
\newtheorem{conjecture}{Conjecture}

\newcommand{\D}{\mathrm{d}}
\newcommand{\A}{\mathcal{A}}
\newcommand{\LL}{\llangle[\Big]}
\newcommand{\RR}{\rrangle[\Big]}
\newcommand{\LD}{\Big\langle}
\newcommand{\RD}{\Big\rangle}
\newcommand{\C}{\mathbb{C}}
\newcommand{\F}{\mathcal{F}}
\newcommand{\HH}{\mathcal{H}}
\newcommand{\X}{\mathcal{X}}
\newcommand{\Z}{\mathbb{Z}}
\newcommand{\E}{\mathcal{E}}
\newcommand{\PP}{\mathbb{P}}
\newcommand{\K}{\mathscr{K}}
\newcommand{\Q}{\mathfrak{Q}}
\newcommand{\q}{\mathbf{q}}

\numberwithin{equation}{section}

\makeatletter
\newcommand{\subjclass}[2][2010]{%
  \let\@oldtitle\@title%
  \gdef\@title{\@oldtitle\footnotetext{#1 \emph{Mathematics Subject Classification.} #2}}%
}
\newcommand{\keywords}[1]{%
  \let\@@oldtitle\@title%
  \gdef\@title{\@@oldtitle\footnotetext{\emph{Key words and phrases.} #1.}}%
}
\makeatother

\title{\bf LG/CY Correspondence for Elliptic Orbifold Curves via Modularity}
\author{Yefeng Shen and Jie Zhou}
\date{}
\maketitle

\begin{abstract}
We prove the Landau-Ginzburg/Calabi-Yau correspondence between the Gromov-Witten theory of each elliptic orbifold curve and its Fan-Jarvis-Ruan-Witten theory counterpart via modularity. We show that the correlation functions in these two enumerative theories are different representations of the same set of quasi-modular forms, expanded around different points on the upper-half plane. We relate these two representations by the Cayley transform.
\end{abstract}

\setcounter{tocdepth}{2} \tableofcontents

\section{Introduction}

Landau-Ginzburg/Calabi-Yau correspondence (LG/CY correspondence for short) is a duality originating from physics \cite{Witten:1993} between the Landau-Ginzburg (LG) model and the non-linear $\sigma$-model defined from the same pair of data $(W,G)$. 
The pairs (W, G) considered in this paper satisfy the following three conditions:
\begin{enumerate}
\item[(1)] The polynomial $W$ is the so-called \emph{superpotential} of the LG-model. It is a weighted homogeneous polynomial over $\C^n$
\begin{equation}
W: \C^n\to\C\,.
\end{equation}
The weight of the $i$-th variable on $\C^n$ will be denoted by $q_i$\,.
\item[(2)]
The polynomial $W$ satisfies the {\em Calabi-Yau condition} (or CY condition) that
\begin{equation}\label{cy-condition}
\sum_{i=1}^n q_i=1\,.
\end{equation} 
\item[(3)] 
The group $G$ is a subgroup of the \emph{group of diagonal symmetries} given by
\begin{equation}\label{diag-symm}
{\rm Aut}(W):=\left\{(\lambda_1,\dots,\lambda_n)\in(\mathbb{C}^*)^n\Big\vert\,W(\lambda_1\,x_1,\dots,\lambda_n\,x_n)=W(x_1,\dots,x_n)\right\}\,.
\end{equation}
The group $G$ is required to contain the so-called \emph{exponential grading element}
\begin{equation}\label{expon-grade}
J_W=\left(\exp(2\pi\sqrt{-1}q_1), \cdots, \exp(2\pi\sqrt{-1}q_n)\right)\,.
\end{equation}
\end{enumerate}
We remark that the \emph{central charge} of $W$ is $n-2$, i.e.,
\begin{equation*}
\hat{c}_W:=\sum_{i=1}^n(1-2q_i)=n-2\,.
\end{equation*}
The CY condition \eqref{cy-condition} implies that the hypersurface $X_W$ defined by $\{W=0\}$ is a ($n-2$)-dimensional CY variety in a weighted projective space. Since $G$ acts on $\C^n$ by homothety, it induces an action on $X_W$, with $J_W\in G$ acting trivially. Thus we get the following CY orbifold 
which is a global quotient
\begin{equation}
\X_{(W,G)}:=X_W/\left(G/\langle J_W\rangle\right)\,.
\end{equation}

\subsection*{GW theory and FJRW theory}

The CY side of the LG/CY correspondence is the Gromov-Witten (GW) theory of the orbifold $\X_{(W,G)}$ which studies the intersection theory of the moduli spaces of stable maps from orbifold curves to the target $\X_{(W,G)}$.
While the LG side is the Fan-Jarvis-Ruan-Witten (FJRW) theory of the pair $(W,G)$ \cite{Fan:20071, Fan:20072}. 
It is an enumerative theory which virtually counts the solutions to the Witten equation \cite{Witten:1993} for the pair $(W,G)$
and is the mathematical construction of the LG A-model of $(W,G)$. 
The details of the GW theory and FJRW theory of the targets studied in this paper will be given in Section \ref{secLG}.

For the pair $(W, G)$, both GW theory and FJRW theory come with a graded vector space equipped with a non-degenerate pairing, which we denote by
\begin{equation*}
\left(\mathcal{H}^{\rm GW},\eta^{\rm GW}\right), \quad \left(\mathcal{H}^{\rm FJRW},\eta^{\rm FJRW}\right)\,.
\end{equation*}
Here $\mathcal{H}^{\rm GW}$ is the \emph{Chen-Ruan cohomology} \cite{Chen:2004} of $\X_{(W,G)}$, and $\mathcal{H}^{\rm FJRW}$ is the \emph{FJRW state space} \cite{Fan:20072} of $(W,G)$.
Let $\overline{\mathcal{M}}_{g,k}$ be the Deligne-Mumford moduli space of $k$-pointed stable curves of genus $g$ and $\psi_j\in H^2(\overline{\mathcal{M}}_{g,k}, \mathbb{Q})$ be the $j$-th $\psi$-class.
Let also $\beta$ be an effective curve class in the underlying coarse moduli of $\X_{(W,G)}$, $\{\alpha_j\}$ be elements in $\mathcal{H}^{\rm GW}$ and $\{\gamma_j\}$ be elements in $\mathcal{H}^{\rm FJRW}$. Then one can define the \emph{ancestor GW invariant} $\langle\alpha_1\psi_1^{\ell_1},\cdots,\alpha_k\psi_k^{\ell_k}\rangle_{g,k,\beta}^{\rm GW}$ and the \emph{FJRW invariant} $\langle \gamma_1\psi_1^{\ell_1},\cdots,\gamma_k\psi_k^{\ell_k}\rangle_{g,k}^{\rm FJRW}$ as integrals over $\overline{\mathcal{M}}_{g,k}$\,. See \eqref{gw-inv} and \eqref{fjrw-inv} for the precise definitions.

We parametrize a K\"ahler class $\mathcal{P}\in\mathcal{H}^{\rm GW}$ by $t$ and set $q=e^{t}$. The Divisor Axiom in GW theory allows us to define a \emph{GW correlation function} as a formal $q$-series
\begin{equation}\label{cy-correlation}
\LL\alpha_1\psi_1^{\ell_1},\cdots,\alpha_k\psi_k^{\ell_k}\RR_{g,k}^{\rm GW}(q)=\sum_{\beta}\LD\alpha_1\psi_1^{\ell_1},\cdots,\alpha_k\psi_k^{\ell_k}\RD_{g,k, \beta}^{\X}\ q^{\int_\beta \mathcal{P}}\,.
\end{equation}
Similarly, we parametrize a degree $2$ element $\phi\in\mathcal{H}^{\rm FJRW}$ by $u$ 
and define an \emph{FJRW correlation function} 
\begin{equation}\label{lg-correlation}
\LL\gamma_1\psi_1^{\ell_1},\cdots,\gamma_k\psi_k^{\ell_k}\RR_{g,k}^{\rm FJRW}(u)=\sum_{n\geq0}{u^n\over n!}\LD \gamma_1\psi_1^{\ell_1},\cdots,\gamma_k\psi_k^{\ell_k}, \phi, \cdots, \phi\RD_{g,k+n}\,.
\end{equation}
The LG/CY correspondence \cite{Witten:1993, Fan:20072} says that the two enumerative theories should be equivalent under an appropriate transformation.  See \cite{Chiodo:2010, Chiodo:2011, Krawitz:2011, Priddis:2013, Chiodo:2014, Priddis:2014} for progresses on the correspondence for various pairs $(W,G)$ at genus zero and higher genus.

Under certain circumstances, one can use modular forms as a tool to establish the correspondence between the GW correlation function in \eqref{cy-correlation} and the FJRW correlation function in \eqref{lg-correlation}.
In this paper, we address this idea for CY orbifolds of dimension one.

\subsection*{LG/CY correspondence for elliptic orbifold curves}
Any one-dimensional CY orbifold must be of the form
\begin{equation}
\X_{a_1}:=\mathbb{P}^1_{a_1, \cdots, a_m}\,, \quad \quad \sum_{i=1}^{m}\frac{1}{a_i}=m-2\,, \quad a_1\geq a_2\geq \cdots \geq a_m>1\,.
\end{equation}
Its underlying space is the projective space $\mathbb{P}^1$ and it has $m$ orbifold points. Each orbifold point is decorated by some $a_i\in\Z_{\geq2}$ if the isotropy group is the cyclic group $\boldsymbol{\mu}_{a_i}$.
These CY orbifolds are called \emph{elliptic orbifold curves} and there are four such orbifolds in total: $\mathbb{P}^1_{2,2,2,2}, \mathbb{P}^1_{3,3,3}, \mathbb{P}^1_{4,4,2}$ and $\mathbb{P}^1_{6,3,2}$.
Each of them can be realized as a quotient of certain elliptic curve $\mathcal{E}_{r}$
\begin{equation}
\X_r=\E_{r}/\boldsymbol{\mu}_{r}, \quad r=a_1=2,3,4,6\,.
\end{equation}
The relation to the elliptic curves is what underlies the (quasi-) modularity of the GW theory which will be reviewed below. 
These elliptic orbifold curves are among the simplest CY varieties and can serve as some toy models in testing ideas and conjectures related to both geometric and arithmetic aspects of mirror symmetry.\\

The goal of this work is to prove the LG/CY correspondence for the pairs $(W,G)$ for which $\X_{(W,G)}$ is an elliptic orbifold curve using modular forms. 
We shall choose the pairs $(W,G)$ as shown in Tab. \ref{tableWGdata}. Here for the \emph{pillowcase orbifold} $\mathbb{P}^1_{2,2,2,2}$, $G_1$ is generated by $(\sqrt{-1},\sqrt{-1})$ and $(1,-1)\in {\rm Aut}(x_1^4+x_2^4)$.
In general, the data $(W,G)$ realizing an elliptic orbifold curve is not unique, see \cite{Krawitz:2011, Milanov:2012qu, Basalaev:2016} for more choices. 

\begin{table}[h]
  \caption[Data of $(W,G)$ for the LG and CY models]{Data of $(W,G)$ for the LG and CY models}
  \label{tableWGdata}
  \renewcommand{\arraystretch}{1.5} 
\begin{displaymath}
  \begin{tabular}{| c| c | c  |}
    \hline
$r=a_1$& elliptic~orbifold~$\X_r$ &  $(W, G)$\\ \hline
$r=3,4,6$&$ \mathbb{P}^{1}_{a_1,a_2,a_3}$   &$\left(x_{1}^{a_1}+x_{2}^{a_2}+x_{3}^{a_3}, \mathrm{Aut}(W)\right)$\\ \hline
$r=2$&$ \mathbb{P}^{1}_{2,2,2,2}$   &  $\left(x_1^4+x_2^4+x_3^2, G_1\times {\rm Aut}(x_3^2)\right)$ \\
     \hline
  \end{tabular}
\end{displaymath}
\end{table}

The LG/CY correspondence for the elliptic orbifold curves $\mathbb{P}^{1}_{a_1,a_2,a_3}$ is already studied in \cite{Krawitz:2011, Milanov:2011, Milanov:2012qu} using techniques from mirror symmetry, Saito's theory of primitive forms \cite{Saito:1983}, and Givental's higher genus formula \cite{Givental:20012}. 
The correspondence is realized by analytic continuation and quantization formulas \cite{Milanov:2011}. 
However, these techniques do not generalize to the pillowcase orbifold since the mirror singularity would have a non-trivial group action and the corresponding theory of primitive forms has not been developed yet. 

\subsection*{LG/CY correspondence via Cayley transform}

While the previous works \cite{Krawitz:2011, Milanov:2011, Milanov:2012qu} rely heavily on the global and analytic structure of the mirror B-model correlation functions, the present work uses the proof of quasi-modularity for the GW correlation functions in \cite{Shen:2014}, and a thorough study on the expansions of these quasi-modular forms around elliptic points on the upper-half plane.  Our proof of the LG/CY correspondence for all elliptic orbifold curves, including the pillowcase orbifold $\mathbb{P}^1_{2,2,2,2}$, works with the A-model directly.

The crucial ingredient in our approach is the \emph{Cayley transformation} $\mathscr{C}: \widehat{M}(\Gamma)\subseteq C^{\omega}_{\mathbb{H}}\rightarrow C^{\omega}_{\mathbb{D}}$
and its variant
$\mathscr{C}_{\mathrm{hol}}: \widetilde{M}(\Gamma)\rightarrow \mathcal{O}_{\mathbb{D}}$. They are
induced by the Cayley transform $\mathcal{C}:\mathbb{H} \to \mathbb{D}$ from the upper-half plane to the disk: based at an interior point $\tau_{*}$ on the upper-half plane, it is given by
\begin{equation}
\mathcal{C}(\tau)={\tau-\tau_{*}\over {\tau\over \tau_{*}-\bar{\tau}_{*} }-  {\bar{\tau}_{*}\over \tau_{*}-\bar{\tau}_{*} }}\,.
\end{equation}
Here $\mathcal{O}$ and $C^{\omega}$ denote the ring of holomorphic and real analytic 
functions, and
$ \widehat{M}(\Gamma),  \widetilde{M}(\Gamma)$
stand for the ring of almost-holomorphic modular forms and quasi-modular forms \cite{Kaneko:1995} for a modular group $\Gamma<\mathrm{SL}_{2}(\mathbb{Z})$, respectively.

Our main result is the following theorem on the LG/CY correspondence for elliptic orbifold curves at the level of correlation functions in all genus.
\begin{thm}\label{main-thm}
[Theorem \ref{lg-cy-main} below]
Let $(W,G)$ be a pair in Tab. \ref{tableWGdata}. Then there exists a degree and pairing preserving isomorphism between the graded vector spaces
\begin{equation}
\mathscr{G}: \left(\mathcal{H}^{\rm GW},\eta^{\rm GW}\right)\to \left(\mathcal{H}^{\rm FJRW},\eta^{\rm FJRW}\right)
\end{equation}
and a Cayley transformation $\mathscr{C}_{\rm hol}$, based at an elliptic point $\tau_{*}\in \mathbb{H}$,  such that for any $\{\alpha_j\}\subseteq \mathcal{H}^{\rm GW},$
\begin{equation}
\mathscr{C}_{\rm hol}\left(\LL\alpha_1\psi_1^{\ell_1},\cdots,\alpha_k\psi_k^{\ell_k}\RR_{g,k}^{\rm GW}(q)\right)=
\LL\mathscr{G}(\alpha_1)\psi_1^{\ell_1},\cdots, \mathscr{G}(\alpha_k)\psi_k^{\ell_k}\RR^{\rm FJRW}_{g,k}(u)\,.
\end{equation}
\end{thm}
The explicit map $\mathscr{G}$ depends on the specific model and will be given in Section \ref{secLGCYcorrespondence}. 
We now explain the main idea of the proof in a few steps. 
\begin{itemize}
\item
According to \cite{Krawitz:2011}, the WDVV equations and $g$-reduction technique work for both the GW theories of elliptic orbifold curves and their FJRW companions. 
These techniques allow us to write both GW and FJRW correlation functions as polynomials in some genus zero correlation functions, which we call \emph{building blocks}. 
\item
In \cite{Shen:2014}, the authors proved that the GW correlation functions of any elliptic orbifold curve are $q$-expansions of some quasi-modular forms for certain modular group (described in Tab. \ref{tablequasimodularity} below). The proof is based on the observation that the WDVV equations for the GW building blocks coincide with the Ramanujan identities for the corresponding quasi-modular forms, as well as the boundary conditions.
\item
We study the local expansions of the quasi-modular forms which arise as the GW building blocks around some elliptic points on the upper-half plane.
These expansions are called \emph{elliptic expansions} and the elliptic points are singled out according to the specific modular group.
We relate these elliptic expansions to the $q$-expansions near the infinity cusp by the Cayley transformation. 
\item
The elliptic expansions for these quasi-modular forms satisfy equations similar to the Ramanujan identities, also the first few terms of these elliptic expansions can be computed straightforwardly. 
We identify them with the WDVV equations and boundary conditions satisfied by the building blocks in the corresponding FJRW theory.
This then leads to the matching between the building blocks in the GW and FJRW theories via the Cayley transformation.
\item
Theorem \ref{main-thm} is then established since the Cayley transformation turns out to be compatible with the reconstruction process for both GW and FJRW correlation functions in all genus.
\end{itemize}

\subsection*{Outline of the paper}

In Section \ref{secLG} we review briefly the quasi-modularity in the GW theories of elliptic orbifold curves and the basics of FJRW theory. 
Following the method given in \cite{Krawitz:2011}, we use WDVV equations and the boundary conditions to get explicit formulas for the prepotentials of the FJRW theories in terms of the building blocks.

Section \ref{secmodularform} discusses the elliptic expansions of quasi-modular forms. 
We first study the elliptic expansions in general, then specialize to the quasi-modular forms that are involved in Tab. \ref{tablequasimodularity} with the help of the elliptic curve families therein.
The computations on the periods and monodromies, especially the precise constants involved on which the local expansions depend on very sensitively, are relegated to Appendix \ref{appendixB}.

Section \ref{secLGCYcorrespondence} is devoted to comparing the WDVV equations and boundary conditions for the two enumerative theories and establishing the LG/CY correspondence using the results obtained in Section \ref{secLG} and \ref{secmodularform}. 

\subsection*{Acknowledgement}
Both authors would like to thank Yongbin Ruan for support.
Y.~S.~would like to thank Amanda Francis, Todor Milanov, Hsian-Hua Tseng and Don Zagier for helpful discussions.
J.~Z.~would like to thank Kevin Costello and Shing-Tung Yau for constant encouragement and support. He also thanks
Boris Dubrovin, Di Yang for useful discussions on Frobenius manifold and integrable hierarchy, and Kathrin Bringmann, Roelof Bruggeman for enlightening discussions and correspondences on elliptic expansions of modular forms.
We also thank the anonymous referees for useful comments which helped improving the paper.

Y.~S.~is partially supported by NSF grant DMS-1159156. J.~Z.~is supported by the Perimeter Institute for Theoretical Physics. Research at Perimeter Institute is supported  by the Government  of Canada through Industry Canada and  by the Province of Ontario through the Ministry of Economic Development and Innovation.

In the middle stage of the present project, we were informed that \cite{Basalaev:2016} was also working on the LG/CY correspondence for elliptic orbifold curves. We thank the authors for correspondences.

\section{WDVV equations in GW theory and FJRW theory}
\label{secLG}

\subsection{Quasi-modularity in GW theories of elliptic orbifold curves}

We first recall the basics of GW theories of elliptic orbifold curves following the exposition part in \cite{Shen:2014}.

Let $H^*_{\rm CR}(\X, \mathbb{C})$ be the Chen-Ruan cohomology of an elliptic orbifold curve $\X$. It is equipped with the pairing $\eta(\cdot ,\cdot)$ and the Chen-Ruan product $\bullet$\,. Moreover,
\begin{equation}
H^*_{\rm CR}(\X, \mathbb{C})=\left(\bigoplus_{i=1}^{m}\bigoplus_{j=1}^{a_i-1}\C\Delta_i^{j}\right)\bigoplus\C\mathbf{1}\bigoplus\C\mathcal{P}\,.
\end{equation}
Here $\Delta_i$ represents the Poincar\'e dual of the fundamental class of the $i$-th orbifold point, with its degree shifted by $2/a_i$; $\Delta_i^j$ is ordinary cup product of $j$-copies of $\Delta_i$; $\bf{1}$ is the Poincar\'e dual of the fundamental class; $\mathcal{P}$ is the Poincar\'e dual of the point class.

Let $\overline{\mathcal{M}}_{g,k, \beta}^{\X}$ be the moduli space of orbifold stable maps $f$ from a genus-$g$ $k$-pointed orbifold curve $C$ to $\X$, with degree $\beta:=f_*[C]\in H_2(\X, \mathbb{Z})$. It has a virtual fundamental class which we denote by $\left[\overline{\mathcal{M}}_{g,k, \beta}^{\X}\right]^{\rm vir}$. Let  $\pi$ be the forgetful morphism to $\overline{\mathcal{M}}_{g,k}$ and $\{{\rm ev}_j\}_{j=1}^{k}$ be the evaluation morphisms to the inertial orbifold.  Then the \emph{ancestor orbifold GW invariant} of $\X$ (called \emph{GW correlator}) is defined to be the integral of the following GW class \cite{Chen:2001, Abramovich:2006}, 
\begin{equation}\label{gw-inv}
\LD \alpha_1\psi_1^{\ell_1},\cdots,\alpha_k\psi_k^{\ell_k}\RD_{g,k,\beta}^{\X}
:=\int_{\left[\overline{\mathcal{M}}_{g,k, \beta}^{\X}\right]^{\rm vir}}
\prod_{j=1}^k{\rm ev}_j^*(\alpha_j)\prod_{j=1}^k\pi^*\left(\psi_j^{\ell_j}\right)\,,
\end{equation}
where  $\alpha_j\in H^*_{\rm CR}(\X, \mathbb{C}),j=1,2,\cdots k$.
We define the \emph{ancestor GW correlation function} as
\begin{equation}
\LL\alpha_1\psi_1^{\ell_1},\cdots,\alpha_k\psi_k^{\ell_k}\RR_{g,k}^{\X}
:=\sum_{n\geq 0}{1\over n!}\sum_{\beta}\LD \alpha_1\psi_1^{\ell_1},\cdots,\alpha_k\psi_k^{\ell_k}, t \mathcal{P}, \cdots, t\mathcal{P}\RD^{\X}_{g,k+n, \beta}\\
\,.
\end{equation}
It is identical to the formula in \eqref{cy-correlation} due to the Divisor Axiom.
The GW invariants give rise to various structures on $H^*_{\rm CR}(\X, \mathbb{C})$.
Among them the \emph{quantum multiplication} $\bullet_{q}$ is defined by 
\begin{equation}
\alpha_1\bullet_{q}\alpha_2=\sum_{\mu, \nu}\LL\alpha_1,\alpha_2,\mu\RR^{\X}_{0,3}\eta^{(\mu, \nu)}\nu\,.
\end{equation}
Here both $\mu, \nu$ range over a basis of $H^*_{\rm CR}(\X, \mathbb{C})$ and $\eta^{(\cdot ,\cdot)}$ is the inverse of the pairing $\eta(\cdot ,\cdot)$.
At the large volume limit $t=-\infty$ (or equivalently $q=0$), the quantum multiplication $\bullet_{q}$ becomes the Chen-Ruan product, i.e., $\bullet_{q=0}=\bullet$\,.

\subsubsection*{WDVV equations and quasi-modularity}

For elliptic orbifold curves, the $q$-series in \eqref{cy-correlation} are proved to be quasi-modular forms \cite{Milanov:2011, Satake:2011, Shen:2014}. The method in \cite{Milanov:2011} uses mirror symmetry and Givental's formalism. It works for all the three cases except for the pillowcase orbifold $\mathbb{P}^1_{2,2,2,2}$.
The method in \cite{Shen:2014}, reviewed below, does not rely on mirror symmetry and applies to all cases.

Let us now recall the WDVV equations.
Let $S$ be a set of indices and $A,B$ be a partition of $S$, that is, $S=A\cup B$ with $A\cap B=\emptyset$. Denote the cardinality of $S$ by $|S|$. There is a forgetful morphism $p: \overline{\mathcal{M}}_{0, |S|+4}\to \overline{\mathcal{M}}_{0, 4}$ which forgets the first $|S|$-markings of a stable curve and then contracts all the unstable components. We index the remaining four points by $i, j, k$ and $\ell$. One can pull back the homologically identical boundary classes from $\overline{\mathcal{M}}_{0, 4}$ to $\overline{\mathcal{M}}_{0, |S|+4}$ and then get the identical boundary classes shown schematically as follows

\begin{equation}\label{wdvv-graph}
\begin{picture}(50,20)


\put(27, 8){$\Longleftrightarrow$}
	\put(50,9){\line(-3,4){5}}
    \put(50,9){\line(-3,-4){5}}
	\put(50,9){\line(1,0){10}}
	\put(60,9){\line(1,0){10}}
    \put(70,9){\line(3,4){5}}
    \put(70,9){\line(3,-4){5}}

	
	\put(43,17){$i$}
	\put(43,0){$k$}
	\put(75,17){$j$}
	\put(75,0){$\ell$}

    \put(50,9){\line(-3,2){5}}
    \put(50,9){\line(-3,-2){5}}
	\put(43,8){$\vdots$}

	\put(37,8){$A\bigg\{$}
	\put(77,8){$\bigg\} B$}

    \put(70,9){\line(3,2){5}}
    \put(70,9){\line(3,-2){5}}
	\put(75,8){$\vdots$}

	\put(-10,9){\line(-3,4){5}}
    \put(-10,9){\line(-3,-4){5}}
	\put(-10,9){\line(1,0){10}}
	\put(0,9){\line(1,0){10}}
    \put(10,9){\line(3,4){5}}
    \put(10,9){\line(3,-4){5}}

	
	\put(-17,17){$i$}
	\put(-17,0){$j$}
	\put(15,17){$k$}
	\put(15,0){$\ell$}

	\put(-23,8){$A\bigg\{$}
	\put(17,8){$\bigg\} B$}

\put(-6,10){$\mu$}
\put(5,10){$\nu$}

\put(54,10){$\mu$}
\put(65,10){$\nu$}

    \put(-10,9){\line(-3,2){5}}
    \put(-10,9){\line(-3,-2){5}}
	\put(-17,8){$\vdots$}

    \put(10,9){\line(3,2){5}}
    \put(10,9){\line(3,-2){5}}
	\put(15,8){$\vdots$}

\end{picture}
\end{equation}

Integrating the GW classes over these identical boundary classes gives a system of equations among the correlators and correlation functions. These equations, called WDVV equations in GW theory, reflect the associativity of the quantum multiplication. More explicitly, one has
\begin{equation}\label{wdvv}
\begin{aligned}
\sum_{A\cup B=S}\,\sum_{\mu,\nu}\,& \sum_{\beta_A+\beta_B=\beta} \LD \alpha_i,\alpha_j, \alpha_A, \mu\RD_{0, |A|+3, \beta_A}\,  \eta^{(\mu,\nu)}\, \LD \nu, \alpha_B, \alpha_k, \alpha_\ell\RD_{0, |B|+3, \beta_B}\\
&=\sum_{A\cup B=S}\,\sum_{\mu,\nu}\, \sum_{\beta_A+\beta_B=\beta}\LD \alpha_i,\alpha_k, \alpha_A, \mu\RD_{0, |A|+3, \beta_A}\, \eta^{(\mu,\nu)}\, \LD \nu, \alpha_B, \alpha_j, \alpha_\ell\RD_{0, |B|+3, \beta_B}\,.
\end{aligned}
\end{equation}
Here the notation $\alpha_A$ means a list $\alpha_a, a\in A$ and the others are similar.

It is shown in \cite{Shen:2014} that the system of WDVV equations for the building blocks for the elliptic orbifold curve $\X_r$ is equivalent to the system of Ramanujan identities ($\partial_{\tau}:={1\over 2\pi i}{\partial\over \partial \tau}$):
\begin{equation}\label{ramanujan}
\left\{
\begin{aligned}
\partial_{\tau}A_{N}&={1\over 2r}A_{N}(E_{N}+{2C_{N}^{r}-A_{N}^{r}\over A_{N}^{r-2}})\,,\\
\partial_{\tau}B_{N}&={1\over 2r}B_{N}(E_{N}-A_{N}^{2})\,,\\
\partial_{\tau}C_{N}&={1\over 2r}C_{N}(E_{N}+A_{N}^{2})\,,\\
\partial_{\tau}E_{N}&={1\over 2r}(E_{N}^{2}-A_{N}^{4})\,.
\end{aligned}
\right.
\end{equation}
Here $(N, r)=(1^{*},6), (2,4), (3,3), (4,2)$ and $A_N, B_N, C_N, E_N$ are some specific quasi-modular forms for the the modular group $\Gamma_{0}(N)<\mathrm{SL}_{2}( \Z)$.
See \cite{Maier:2009, Maier:2011, Zhou:2013hpa} for details on these quasi-modular forms.
The boundary conditions also match.

As a consequence of the existence and uniqueness of solutions to an ODE system, the building block correlation functions are quasi-modular forms. 
More precisely, the quasi-modularity for the GW theories of these elliptic orbifold curves has the pattern \cite{Shen:2014} indicated in Tab. \ref{tablequasimodularity}. The elliptic curve families shown in the table have the origin from the mirror singularities \cite{Milanov:2011, Krawitz:2011, Satake:2011}. See also the recent works \cite{Cho:2014, Lau:2014, Cho:2015} for a geometric construction of these curve families using Floer theory and SYZ mirror symmetry.

\begin{table}[h]
  \caption[Quasi-modularity in GW theories of elliptic orbifold curves]{Quasi-modularity in GW theories of elliptic orbifold curves}
  \label{tablequasimodularity}
  \renewcommand{\arraystretch}{1.5} 
\begin{displaymath}
  \begin{tabular}{ | c | c c  c c| }
    \hline
 elliptic~orbifold~$\mathcal{X}_{r}=\mathcal{X}_{(W,G)}$ &  $ \mathbb{P}^{1}_{2,2,2,2}$ &  $ \mathbb{P}^{1}_{3,3,3}$ &  $ \mathbb{P}^{1}_{4,4,2}$ &  $ \mathbb{P}^{1}_{6,3,2}$\\ \hline
  group~action~$\mathbb{Z}_{r}$ in~forming~the~orbifold &  $r=2$ &  $3$ &  $4$ &  $6$\\
modular group~$\Gamma(r)$ for correlation functions &$ \Gamma(2) $&$\Gamma(3)$& $\Gamma(4)$& $\Gamma(6)$\\
elliptic~curve~family&  $ D_{4}$ &  $  E_{6}$ &  $ E_{7}$ &  $E_{8}$\\
 modular~group~$\Gamma_{0}(N)$ for~elliptic~curve~family&   $ \Gamma_{0}(4)$ &  $\Gamma_{0}(3)$ &  $ \Gamma_{0}(2)$ &  $\Gamma_{0}(1^{*})$\\
     \hline
  \end{tabular}
\end{displaymath}
\end{table}

Moreover, using WDVV equations further and tautological relations in higher genus (including Getzler's relation and the $g$-reduction technique), all the non-vanishing correlation functions in \eqref{cy-correlation} can be expressed as polynomials of the building block correlation functions and hence are quasi-modular forms. This will be further explained in Section \ref{secLGCYcorrespondence}.

\subsection{Review on FJRW theory}

In this section we first review the basic ingredients of FJRW theory \cite{Fan:20071, Fan:20072}. Then we discuss the WDVV equations and boundary conditions for the FJRW theories of the pairs in Tab. \ref{tableWGdata}.

The main ingredient in the FJRW theory for a pair $(W,G)$ is a \emph{Cohomological Field Theory} (CohFT in short,  in the sense of \cite{Kontsevich:1994}) on an \emph{FJRW state space} $\HH_{(W,G)}$.
Explicitly,
\begin{equation}\label{fjrw-state}
\HH_{(W,G)}:=\bigoplus_{h\in G}\HH_{h} \quad \text{with} \quad \HH_h:=(H^{N_h}({\rm Fix}(h), W_h^{\infty}; \C))^{G}\,.
\end{equation}
Here ${\rm Fix}(h)\subseteq\C^n$ is the fixed locus of the element $h\in G$ and is a space of complex dimension $N_h$, $W_h^{\infty}:=({\rm Re}\, W|_{{\rm Fix}(h)})^{-1} (-\infty, M\ll 0)$, where ${\rm Re}\, W|_{{\rm Fix}(h)}$ is the real part of $W|_{{\rm Fix}(h)}$.

If ${\rm Fix}(h)\neq 0$, the elements in $\HH_h$ will be called \emph{broad}.
If ${\rm Fix}(h)=0\in\C^n$, then $N_h=0$, and 
$\HH_h=H^{0}(\{0\}, \emptyset; \C)\cong\C$. Its elements are called \emph{narrow}. In this case, there is a canonical choice 
\begin{equation}
\phi_h=1\in H^{0}(\{0\}, \emptyset; \C)\,.
\end{equation}
The state space $\HH_{(W,G)}$ is a graded vector space. For each $h\in G$, there exist unique $\{\Theta_h^{(i)}\in [0,1)\cap\mathbb{Q}\}_{i=1}^{n}$, such that 
\begin{equation}
h=\left(\exp(2\pi\sqrt{-1}\Theta_h^{(1)}), \cdots, \exp(2\pi\sqrt{-1}\Theta_h^{(n)})\right)\,.
\end{equation}
For a homogeneous element $\gamma\in\HH_h$, its degree is defined by
\begin{equation}\label{fjrw-deg}
\deg_W\gamma:=\frac{N_h}{2}+\sum_{i=1}^{n}(\Theta_h^{(i)}-q_i)\,.
\end{equation}

\subsubsection*{FJRW invariants}
 The CohFT for an FJRW theory consists of multi-linear maps 
\begin{equation}
\Lambda_{g,k}^{(W,G)}: (\HH_{W,G})^{\otimes k}\to H^*(\overline{\mathcal{M}}_{g,k}, \mathbb{C})\,.
\end{equation}
The construction of $\{\Lambda_{g,k}^{(W,G)}\}_{g,k}$ is highly non-trivial. It involves solving the Witten equation as well as constructing a virtual fundamental cycle analytic over the moduli space of solutions \cite{Fan:20071}. 
See also \cite{Chang:2013} for an algebro-geometric construction for the narrow elements via the cosection technique.

Let $\{\gamma_j\}\subseteq \HH_{(W,G)}$, and $\psi_j\in H^*(\overline{\mathcal{M}}_{g,k}, \mathbb{C})$ be the $j$-th $\psi$-class, the following integral defines a genus-$g$ $k$-point \emph{FJRW invariant} (or \emph{FJRW correlator}),
\begin{equation}\label{fjrw-inv}
\LD\gamma_1\psi_1^{\ell_1},\cdots,\gamma_k\psi_k^{\ell_k}\RD_{g,k}^{(W,G)}
:=\int_{\overline{\mathcal{M}}_{g,k}}\Lambda_{g,k}^{(W,G)}(\gamma_1,\cdots,\gamma_k) \prod_{j=1}^k\psi_j^{\ell_j}\,.
\end{equation}
The invariant is called {\em primary} if $\ell_j=0$ for all $1\leq j\leq k$. 
The class $\Lambda_{g,k}^{(W,G)}(\gamma_1,\cdots,\gamma_k)$ has a degree which depends on the central charge $\hat{c}_W$, the genus $g$, and $\{\deg_W\phi_j\}$. The correlator in \eqref{fjrw-inv} is nonzero only if the following condition (called \emph{Degree Axiom}) holds
\begin{equation}\label{degree-axiom}
\hat{c}_W(g-1)+\sum_{j=1}^{k}\deg_W\gamma_j+\sum_{j=1}^{k}\ell_j=3(g-1)+k\,.
\end{equation}
Let $\gamma_j\in\HH_{h_j},j=1,2\cdots k$, if the correlator in \eqref{fjrw-inv} is nonzero, then the following \emph{Selection Rule} holds
\begin{equation}\label{selection}
q_i(2g-2+k)-\sum_{j=1}^{k}\Theta^{(i)}_{h_{j}}\in \Z\,,\quad i=1,2\cdots n\,.
\end{equation}
We refer the readers to \cite{Fan:20072} for a complete list of axioms that these FJRW invariants satisfy. 
As a consequence, the FJRW invariants induce various structures on the state space $\HH_{(W,G)}$, including:
\begin{itemize}
\item[(1)] A Frobenius algebra $(\HH_{(W,G)}, \bullet)$, where the multiplication $\bullet$ is defined from the pairing $\eta^{\rm FJRW}(\cdot,\cdot)$ on $\HH_{(W,G)}$ and the genus zero $3$-point invariants through the following formula
\begin{equation}\label{fjrw-ring}
\eta^{\rm FJRW}(\gamma_1, \gamma_2\bullet\gamma_3)=\LD\gamma_1,\gamma_2,\gamma_3\RD_{0,3}^{(W, G)}\,.
\end{equation}
\item[(2)] A formal Frobenius manifold structure on $\HH_{(W,G)}$. Fixing a basis $\{\phi_j\}$ of $\HH_{(W,G)}$, the prepotential $\mathcal{F}_{0,(W,G)}^{\rm FJRW}$ is assembled from all genus zero primary FJRW invariants
\begin{equation}\label{fjrw-pot}
\mathcal{F}_{0,(W,G)}^{\rm FJRW}=\sum_{k\geq3}
\frac{1}{k!}\LD{\bf u}_0,\cdots,{\bf u}_0\RD_{0,k}^{(W,G)}, \quad {\bf u}_0=\sum_j  u_j\phi_j\,.
\end{equation}

\item[(3)]  Let ${\bf u}=\big({\bf u}_0, {\bf u}_1, \cdots\big)$, where 
 \begin{equation}
 {\bf u}_a=\sum_{j}\, u_{j, a}\, \phi_j\, z^a\,.
 \end{equation}
 The \emph{total ancestor potential} $\mathcal{A}_{(W,G)}^{\rm FJRW}$ is the generating series assembled from the FJRW invariants of all genera:
\begin{equation}\label{fjrw-an}
\mathcal{A}_{(W,G)}^{\rm FJRW}=\exp\left(\sum_{g\geq0}\hbar^{1-g}\sum_{k; 2g-2+k>0}\frac{1}{k!}\LD{\bf u}(\psi),\cdots,{\bf u}(\psi)\RD_{0,k}^{(W,G)}\right)\,.
\end{equation}
\end{itemize}

For simplicity, we shall occasionally drop the super- and sub-scripts like $(g,k)$ and $(W,G)$ when they are clear from the surrounding texts.

\subsubsection*{Quantum product and derivatives}

For an element $\phi\in \HH_{(W,G)}$ with $\deg_W\phi=1$, we parametrize it by $u$ and introduce the following formal power series (primary correlation function, see \eqref{lg-correlation})
\begin{equation}\label{translation}
\LL \gamma_1,\cdots,\gamma_k\RR_{g,k} =\sum_{n\geq0}{u^n\over n!}\LD \gamma_1,\cdots,\gamma_k,\phi, \cdots, \phi\RD_{g,k+n}\,.
\end{equation}
Its derivative is again a formal power series
\begin{equation}\label{derivative-fjrw}
\begin{aligned}
\frac{\rm d}{{\rm d}u}\LL \gamma_1,\cdots,\gamma_k\RR_{g,k}
&
=\sum_{n\geq1}{u^{n-1}\over (n-1)!}\LD \gamma_1,\cdots,\gamma_k,\phi, \cdots, \phi\RD_{g,k+1+n-1}\\
&=\LL \gamma_1,\cdots,\gamma_k, \phi\RR_{g,k+1}
\,.
\end{aligned}
\end{equation}
Let $\eta^{(\cdot, \cdot)}$ be the inverse of the paring $\eta^{\rm FJRW}(\cdot, \cdot)$. One can consider a product $\bullet_u$ defined by
\begin{equation}
\gamma_1\bullet_u\gamma_2=\sum_{\gamma, \xi}\LL\gamma_1,\gamma_2,\gamma\RR_{0,3}\eta^{(\gamma, \xi)}\xi\,.
\end{equation}
Clearly this \emph{quantum product} $\bullet_u$ is a deformation of the multiplication $\bullet$ in \eqref{fjrw-ring} as $\bullet_{u=0}=\bullet$\,.

\subsection{WDVV equations in FJRW theories}

Now we discuss the WDVV equations in the FJRW theories for the pairs in Tab. \ref{tableWGdata}.
We restrict ourselves to the cases when the elliptic orbifold curve is $\mathbb{P}^{1}_{3,3,3}$ or $\mathbb{P}^{1}_{2,2,2,2}$. The other two cases are similar to the $\mathbb{P}^{1}_{3,3,3}$ case.
The following definition will be useful later.
\begin{dfn}
Given a homogeneous basis $\{\gamma_j\}$ of $\HH_{(W,G)}$, an element $\gamma\in \HH_{(W,G)}$ is called \emph{primitive} if it cannot be written as $\gamma= \gamma_k \bullet \gamma_\ell$ with $\deg_W\gamma_k, \deg_W\gamma_\ell >0$. A primary correlator $\langle  \cdots \rangle_{0, k}$ (or a primary correlation function $\langle \langle  \cdots \rangle \rangle_{0, k}$) is called \emph{basic} if at least $k-2$ insertions are primitive. 
\end{dfn}

\subsubsection*{Cubic case}
We first consider the $r=3$ case in Tab. \ref{tableWGdata} with
\begin{equation}\label{fjrw-cubic}
\left(W=x_1^3+x_2^3+x_3^3, \quad G={\rm Aut}(W)\cong\boldsymbol{\mu}_{3}\times \boldsymbol{\mu}_{3}\times \boldsymbol{\mu}_{3} \right)\,.
\end{equation}
Let $\omega=\exp(2\pi \sqrt{-1}/3)$. By the definition given in \eqref{fjrw-state}, it is easy to see that each of the following elements in ${\rm Aut}(W)$ generates an one-dimensional subspace $\HH_{h_i}$ of $\HH_{(W,G)}$:
\begin{equation*}
\left\{
\begin{array}{llll}
h_0=(\omega,\omega,\omega),&
h_1=(\omega^2,\omega,\omega),&
h_2=(\omega,\omega^2,\omega),&
h_3=(\omega,\omega,\omega^2)\\
h_7=(\omega^2,\omega^2,\omega^2),&
h_6=(\omega,\omega^2,\omega^2),&
h_5=(\omega^2,\omega,\omega^2),&
h_4=(\omega^2,\omega^2,\omega)
\end{array}
\right.
\end{equation*}
We now pick $\phi_i=1\in H^0({\rm Fix}(h_i))$ and fix a basis $\{\phi_i\}_{i=0}^{7}$. The non-degenerate pairing is given by
\begin{equation*}
\eta^{\rm FJRW}(\phi_i, \phi_j)=\delta_{i+j\,,\,7}\,.
\end{equation*}
The non-trivial relations are
\begin{eqnarray*}
&
\phi_i\bullet\phi_j=\phi_{i+j+1} \quad \text{and} \quad \phi_i\bullet\phi_j\bullet\phi_k=\phi_7 \quad \text{if} \quad \{i, j, k\}=\{1,2,3\}\,,\nonumber\\
&\phi:=\phi_7=\phi_k\bullet\phi_{7-k}, \quad \quad 0\leq k\leq7\,.
\end{eqnarray*}
Here as before we parametrize $\phi$ by $u$, and also parametrize $\phi_{i}$ by $u_{i}$, $i=0,\cdots 6$.
\begin{rem}
In this case, by mapping $\phi_i$ to $x_i$,  $i=1,2,3$, one gets a ring isomorphism \cite{Krawitz:2011}  from the state space to the Jacobi algebra of $W$
\begin{equation*}
\left(\HH_{(W, {\rm Aut}(W))}, \bullet\right)\cong {\rm Jac}(W):=\C[x_1,x_2,x_3]\bigg/\left(\frac{\partial W}{\partial x_1}, \frac{\partial W}{\partial x_2}, \frac{\partial W}{\partial x_3}\right)\,.
\end{equation*}
\end{rem}

Since $\deg_W\phi=1$, according to \eqref{translation} one can define the following basic correlation functions:
\begin{equation}\label{cubic-block}
\left\{
\begin{array}{lll}
f_1(u):=\LL \phi_1,\phi_2,\phi_3\RR, &
f_2(u):=\LL \phi_1,\phi_1,\phi_1\RR,&
f_3(u):=\LL \phi_1,\phi_1,\phi_6,\phi_6\RR, 
\\
f_4(u):=\LL \phi_1,\phi_2,\phi_4,\phi_4\RR, &
f_5(u):=\LL \phi_1,\phi_1,\phi_4,\phi_5\RR, &
f_6(u):=\LL \phi_1,\phi_2,\phi_5,\phi_6\RR\,.
\end{array}
\right.
\end{equation}
We remark that the correlation functions are invariant under the action of the symmetric group $\mathfrak{S}_3$ on the ordered pairs $\{(\phi_1, \phi_6), (\phi_2, \phi_5), (\phi_3,\phi_4)\}$.

\begin{prop}\label{fjrw-e3}
The prepotential of the FJRW theory for the pair $(W, G)$ in \eqref{fjrw-cubic} is given by 
\begin{equation}\label{eqncubicF}
\begin{aligned}
\mathcal{F}^{\rm FJRW}_{0,(W,G)}
=&\frac{1}{2} u_0^2u
+\frac{1}{3} u_0(u_1u_6+u_2u_5+u_3u_4)
+(u_1u_2u_3) \, f_{1}
+\frac{1}{6} (u_1^3+u_2^3+u_3^3) \, f_{2}
\\
+&(u_1u_2u_5u_6+u_1u_3u_4u_6+u_2u_3u_4u_5)\, \frac{3f_{3}+f_{2}^2}{6}
+\frac{1}{2} (u_1^2u_4u_5+u_2^2u_4u_6+u_3^2u_5u_6) \, {f_{1}^2\over 3}
\\
+&\frac{1}{2} (u_1u_2u_4^2+u_1u_3u_5^2+u_2u_3u_6^2) \, { f_{1}f_{2}\over 3}
+\frac{1}{4} (u_1^2u_6^2+u_2^2u_5^2+u_3^2u_4^2) \, f_{3}
\\
+&\frac{1}{2} (u_1u_4u_5u_6^2+u_2u_4u_5^2u_6+u_3u_4^2u_5u_6) \, { f_{1}^2 f_{2}  \over 9}
+\frac{1}{4} (u_1u_4^2u_5^2+u_2u_4^2u_6^2+u_3u_5^2u_6^2) \, {   f_{1} f_{2}^2\over 9}
\\
+&\frac{1}{6} (u_1u_6(u_4^3+u_5^3)+u_2u_5(u_4^3+u_6^3)+u_3u_4(u_5^3+u_6^3))  \, { f_{1}^3\over 9}
+\frac{1}{24} (u_1u_6^4+u_2u_5^4+u_3u_4^4) \, { f_{2}^3\over 9}
\\
+&\frac{1}{8}(u_4^2u_5^2u_6^2) \, {2 f_{1}^4+  f_{1} f_{2}^3 \over 27}
+\frac{1}{36} (u_4^3u_5^3+u_4^3u_6^3+u_5^3u_6^3) \, { f_{1}^3  f_{2}\over 9}
\\
+&\frac{1}{24} (u_4u_5u_6^4+u_4u_5^4u_6+u_4^4u_5u_6)  \, { f_{1}^2 f_{2}^2  \over 9}
+\frac{1}{720}(u_4^6+u_5^6+u_6^6) \, {2f_{1}^3 f_{2}-f_{2}^4\over 9}\,.
\end{aligned}
\end{equation}
Here $f_1, f_2, f_3$ are the correlation functions given in \eqref{cubic-block}.
\end{prop}
\begin{proof}
We study the WDVV equations for these correlation functions
basing on the identity illustrated in \eqref{wdvv-graph}. For example, we can choose $(i,j,k,\ell)=(2,1,3,6)$ and $S=\{1\}$. Omitting the diagrams which have no contribution, we obtain 
\begin{center}
\setlength{\unitlength}{0.08cm}
\begin{picture}(60,23)

\put(27, 8){$=$}
\put(83, 8){$+$}
\put(-27, 8){$+$}

	\put(-60,9){\line(-3,4){5}}
    \put(-60,9){\line(-3,-4){5}}
	\put(-60,9){\line(1,0){10}}
	\put(-50,9){\line(1,0){10}}
    \put(-40,9){\line(3,4){5}}
    \put(-40,9){\line(3,-4){5}}

	
	\put(-67,17){$\phi_2$}
        \put(-70,8){$\phi_1$}
	\put(-67,0){$\phi_1$}
        \put(-56,10){$\mu$}
        \put(-45,10){$\nu$}
	\put(-35,17){$\phi_3$}
	\put(-35,0){$\phi_6$}

    \put(-60,9){\line(-2,0){5}}

	\put(50,9){\line(-3,4){5}}
    \put(50,9){\line(-3,-4){5}}
	\put(50,9){\line(1,0){10}}
	\put(60,9){\line(1,0){10}}
    \put(70,9){\line(3,4){5}}
    \put(70,9){\line(3,-4){5}}

	
	\put(43,17){$\phi_2$}
        \put(40,8){$\phi_1$}
	\put(43,0){$\phi_3$}
        \put(54,10){$\mu$}
        \put(65,10){$\nu$}
	\put(75,17){$\phi_1$}
	\put(75,0){$\phi_6$}

    \put(50,9){\line(-2,0){5}}

	\put(100,9){\line(-3,4){5}}
    \put(100,9){\line(-3,-4){5}}
	\put(100,9){\line(1,0){10}}
	\put(110,9){\line(1,0){10}}
    \put(120,9){\line(3,4){5}}
    \put(120,9){\line(3,-4){5}}

	
	\put(93,17){$\phi_2$}
	\put(93,0){$\phi_3$}
        \put(104,10){$\phi_1$}
        \put(115,10){$\phi_6$}
	\put(125,17){$\phi_1$}
        \put(126,8){$\phi_1$}
	\put(125,0){$\phi_6$}

    \put(120,9){\line(2,0){5}}

	\put(-10,9){\line(-3,4){5}}
    \put(-10,9){\line(-3,-4){5}}
	\put(-10,9){\line(1,0){10}}
	\put(0,9){\line(1,0){10}}
    \put(10,9){\line(3,4){5}}
    \put(10,9){\line(3,-4){5}}

	
	\put(-17,17){$\phi_2$}
	\put(-17,0){$\phi_1$}
        \put(-6,10){$\phi_3$}
        \put(5,10){$\phi_4$}
	\put(15,17){$\phi_3$}
        \put(16,8){$\phi_1$}
	\put(15,0){$\phi_6$}

    \put(10,9){\line(3,0){5}}

\end{picture}
\end{center}
We may insert the $\phi$-classes such that their indices belong to $S$ and then sum over the corresponding equations in the way according to the right hand side of \eqref{translation}. Then we get the following equation among the correlation functions:
\begin{eqnarray*}
\LL\phi_1,\phi_2,\phi_1,\phi_3\bullet_u\phi_6\RR&+&
\LL \phi_1,\phi_2\bullet_u\phi_1, \phi_3,\phi_6\RR \nonumber \\
&=&
\LL \phi_1,\phi_2,\phi_3,\phi_1\bullet_u\phi_6\RR+
\LL \phi_1,\phi_2\bullet_u\phi_3,\phi_1,\phi_6\RR\,.
\end{eqnarray*}
Recall that $\phi_6\bullet_u\phi_3=0$, $\phi_k\bullet_u\phi_{7-k}=\phi, 0\leq k\leq7$, we then have
\begin{eqnarray*}
\phi_2\bullet_u\phi_1=\LL \phi_2,\phi_1,\phi_3\RR\ \phi_4=f_1(u)\phi_4\,, \quad \phi_2\bullet_u\phi_3=\LL \phi_2,\phi_3,\phi_1\RR\ \phi_6=f_1(u)\phi_6\,.
\end{eqnarray*} 
Combining with \eqref{derivative-fjrw}, this equation becomes
\begin{equation}
0+f_1(u)f_6(u)=f_1'(u)+f_1(u)f_3(u)\,.
\end{equation}
Here $f_1'(u)$ is the derivative of $f_1(u)$.
Similarly, we can derive the following collection of WDVV equations (the graphical explanations are presented in Appendix \ref{appendixA}):
\begin{equation}\label{wdvv-333}
\left\{
\begin{aligned}
&f_1f_4=f_2f_5, & (i,j,k,\ell)=(1,2,1,2), && S=\{4\},\\
&f_1f_5=f_2'+f_2f_6, &  (i,j,k,\ell)=(1,2,1,5), & & S=\{1\},\\
&2f_1f_6=f_1f_3+f_2f_4, & (i,j,k,\ell)=(1,2,1,3), && S=\{6\},\\
&f_1f_5'=2f_5f_1', & (i,j,k,\ell)=(1,2,5,7), && S=\{1,1\},\\
&f_1f_3'=2f_6f_1', & (i,j,k,\ell)=(2,3,4,7), && S=\{1,1\}.\\
\end{aligned}
\right.
\end{equation}
The boundary conditions can be calculated using the Selection Rule \eqref{selection} 
and Grothendieck-Riemann-Roch formula. As computed in \cite{Krawitz:2011}, one has
\begin{equation}\label{initial-333}
f_2(0)=f_3(0)=f_4(0)=f_6(0)=0, \quad f_1(0)=1, \quad f_2'(0)=f_5(0)={1\over3}\,.
\end{equation}
Using the boundary conditions \eqref{initial-333}, we can solve for $f_4(u), f_5(u)$ and $f_6(u)$ from \eqref{wdvv-333}:
\begin{equation}
f_4= {f_{1}f_{2}\over 3}, \quad f_5= {f_{1}^2\over 3}, \quad f_6=\frac{3f_{3}+f_{2}^2}{6}\,.
\end{equation}
Eliminating the quantity $f_6$, we can rewrite the rest of the equations as
\begin{equation}\label{eqWDVVFJRW333}
\left\{
\begin{aligned}
f_2'&=\frac{1}{6}f_2\Big(-3f_3+\frac{2f_1^3-f_2^3}{f_2}\Big),\\
f_1'&=\frac{1}{6}f_1\Big(-3f_3+f_2^2\Big),\\
(-3f_3)'&=\frac{1}{6}\Big((-3f_3)^2-(f_2^2)^2\Big).
\end{aligned}
\right.
\end{equation}
Following the procedure given in \cite{Krawitz:2011}, we obtain the rest of the correlation functions by analyzing the WDVV equations among all the genus zero primary correlators. 
\end{proof}
It is clear from the above formula \eqref{eqncubicF} for the prepotential that the dependence in $u$ are in polynomials of $f_{1},f_{2},f_{3}$. 
Note also that the above expression assembles the same form as the one for the orbifold GW theory of $\mathbb{P}^1_{3,3,3}$ given in \cite{Satake:2011, Shen:2014}. This observation is one of the motivations for finding the exact matching between the two theories.

\subsubsection*{Pillowcase}

According to \cite[Theorem 4.1.8 (8)]{Fan:20072}, the FJRW theory of the pair $(W,G)$ shown in the $r=2$ case in Tab. \ref{tableWGdata} is essentially equivalent to the FJRW theory of 
\begin{equation}\label{eqnpairforpillowcase}
\left(W_1=x_1^4+x_2^4, \quad G_1=\left\langle J=(\sqrt{-1}, \sqrt{-1}), \sigma=(1, -1)\right\rangle\cong \boldsymbol{\mu}_4\times \boldsymbol{\mu}_2\right)\,.\end{equation}
A basis of $\HH_{(W_1,G_1)}$ is induced from the following set of elements in $G_1$:
\begin{equation*}
h_0=J, \quad h_1=J\sigma, \quad h_2=J^2, \quad h_3=J^3\sigma,\quad h_4=1, \quad h_5=J^3\,.
\end{equation*}
Each element $h_i$ gives an one-dimensional subspace $\HH_{h_{i}}$, which is spanned by
\begin{equation*}
\left\{
\begin{array}{ll}
\phi_i=1\in H^0({\rm Fix}(h_i)), & i\neq4\,,\\
\phi_4=xy\D x\D y\in H^2({\rm Fix}(h_i))\,.&
\end{array}
\right.
\end{equation*}
Note that $\phi_4$ is the only broad element and will occasionally be denoted by $R$ below.
According to the degree formula in \eqref{fjrw-deg}, we have
\begin{equation*}
\deg_W(\phi_0)=0; \quad \deg_W(\phi_5)=1; \quad \deg_{W}(\phi_i)=\frac{1}{2}, \quad i\neq0,5\,.
\end{equation*}
The non-vanishing part of the non-degenerate pairing is given by
\begin{equation*}
\eta^{\rm FJRW} (\phi_0, \phi)=\eta^{\rm FJRW}( \phi_1, \phi_3 )=\eta^{\rm FJRW}( \phi_2, \phi_2)=\eta^{\rm FJRW}( R, R)=1\,.
\end{equation*}
Here $\phi:=\phi_{5}$ is the only degree one element and as before we parametrize it by $u$, we also use the coordinates $u_{i}$ for $\phi_{i},i=0,1,2,3,4$.

The reconstruction of the genus zero primary potential and the boundary conditions are thoroughly discussed in \cite{Francis:2014}. The nonzero primary correlation functions in genus zero must be a linear combination of the following functions (see \cite[Page 17]{Francis:2014})
\begin{equation}\label{pillow-initial}
\left\{
\begin{array}{ll}
g_1:=\LL \phi_1, \phi_1, \phi_2, \phi_2\RR={1\over 4}+\mathcal{O}(u^2), & 
g_2:=\LL \phi_2, \phi_2, \phi_2, \phi_2\RR={u\over 16}+\mathcal{O}(u^3), \\
g_3:=\LL \phi_1, \phi_3, \phi_2, \phi_2\RR={u\over 16}+\mathcal{O}(u^3), &
g_4:=\LL \phi_1, \phi_1, \phi_1, \phi_1\RR={u\over 8}+\mathcal{O}(u^3), \\
g_5:=\LL \phi_1, \phi_1, \phi_1, \phi_3\RR=0+\mathcal{O}(u^3), & 
g_6:=\LL \phi_1, \phi_1, \phi_3, \phi_3\RR=0+\mathcal{O}(u^2), \\
f_1:=\LL R,R, \phi_1, \phi_1\RR=-{1\over 4}+\mathcal{O}(u^2), & 
f_2:=\LL R,R, \phi_2, \phi_2\RR=-{u\over 16}+\mathcal{O}(u^3), \\
f_3:=\LL R, R, \phi_1, \phi_3\RR={u\over 16}+\mathcal{O}(u^3),   & 
f_4:=\LL R, R, R, R\RR={u\over 16}+\mathcal{O}(u^3). 
\end{array}
\right.
\end{equation}
For the correlation functions with broad elements, we obtain the following WDVV equations (see the graphical explanations in Appendix \ref{appendixA})
\begin{equation}\label{pillow-wdvv-1}
\left\{
\begin{aligned}
& f_1'+2f_1f_2=0, & (i,j,k,\ell)=(2,2,1,1), && S=\{4,4\}. \\
&f_2'+2f_2f_3+f_3'=0, & (i,j,k,\ell)=(1,3,2,2), && S=\{4,4\}.\\
&f_3'+f_3^2=f_1^2, & (i,j,k,\ell)=(1,3,1,3), && S=\{4,4\}.\\
\end{aligned}
\right.
\end{equation}
The prepotential is essentially determined from these three correlation functions.
\begin{prop}\label{pillow-fjrw}
The prepotential of the FJRW theory in this case is given by
\begin{equation}
\begin{aligned}
\F_{0,(W,G)}^{\rm FJRW}=&\frac{1}{2}u_0^2 u+u_0\left(u_1u_3+\frac{1}{2}u_2^2+\frac{1}{2}u_4^2\right)+f_1(u)\frac{1}{2!2!}(u_1^2+u_3^2)(u_4^2-u_2^2)+\\
&+f_2(u)\left(\frac{1}{4!}(-u_1^4+u_2^4-u_3^4+u_4^4)+\frac{1}{2!2!}(u_1^2u_3^2+u_2^2u_4^2)\right)+\\
&+f_3(u)\left(\frac{1}{4!}(u_1^4+2u_2^4+u_3^4+2u_4^4)+\frac{1}{2!2!}u_1^2u_3+\frac{1}{2!}u_1u_3(u_2^2+u_4^2)\right).
\end{aligned}
\end{equation}
\end{prop}
\begin{proof}
For $f_4$ and the $g$-type correlation functions, we also have WDVV equations. The explicit equations with the corresponding choices for $(i,j,k,\ell)$ and $S$ are given as follows
\begin{equation*}
\left\{
\begin{aligned}
& f_4'+2f_4f_2+f_2'=2f_2^2, & (i,j,k,\ell)=(2,2,4,4), && S=\{4,4\}. \\
& f_4'+2f_4f_3+f_3'=2f_3^2+2f_1^2, & (i,j,k,\ell)=(1,3,4,4), && S=\{4,4\}. \\
&g_1'+2g_2g_1=4g_1g_3, & (i,j,k,\ell)=(2,2,1,1), && S=\{2,2\}.\\
&g_2'+2g_3g_2+g_3'=2g_1^2+2g_3^2, & (i,j,k,\ell)=(1,3,2,2), && S=\{2,2\}.\\
&g_3'+g_3^2=g_1^2, & (i,j,k,\ell)=(1,3,1,3), && S=\{2,2\}.\\
&g_4'+2g_3g_4+2g_1g_5=2g_1^2, & (i,j,k,\ell)=(2,2,1,1), && S=\{1,1\}. \\
&g_1'+2g_5g_1+2g_5g_3+g_5'=2g_1g_3, & (i,j,k,\ell)=(1,3,2,2), && S=\{1,1\}. \\
&2g_6'=g_4^2-g_6^2, & (i,j,k,\ell)=(1,3,1,3), && S=\{1,3\}. \\
&g_5'+2g_5g_6=g_5g_6+g_5g_4, & (i,j,k,\ell)=(1,3,1,3), && S=\{1,1\}.\\
&g_6'+2g_1g_5+2g_3g_6=2g_3^2, & (i,j,k,\ell)=(2,2,3,3), && S=\{1,1\}.\\
&g_5'+g_1(g_4+g_6)+2g_3g_5=2g_1g_3, & (i,j,k,\ell)=(2,2,3,3), && S=\{1,3\}.
\end{aligned}
\right.
\end{equation*}

Comparing these equations to \eqref{pillow-wdvv-1} and using the boundary conditions in \eqref{pillow-initial}, we get
\begin{equation}\label{solution-pillow}
\left\{
\begin{array}{llll}
g_1=-f_1\,, &
g_2=f_2+2f_3\,,&
g_3=f_3\,, &\\
g_4=-f_2+f_3\,,&
g_5=0\,,&
g_6=f_2+f_3\,,&
f_4=f_2+2f_3\,.
\end{array}
\right.
\end{equation}
Now the result follows by straightforward computations.
\end{proof}

\section{Local expansions of quasi-modular and almost-holomorphic modular forms}
\label{secmodularform}

This section is devoted to the discussion on local expansions of modular forms and of their variants quasi-modular forms and almost-holomorphic modular forms. The results will be used in Section \ref{secLGCYcorrespondence} to establish the LG/CY correspondence.

\subsection{Cayley transformation and definition of local expansion}

Before working out the details for the expansions near a point in the interior of the upper-half plane, we first review some standard material about modular forms following \cite{Zagier:2008}, see also the textbooks \cite{Sch:1974, Rankin:1977ab}. Throughout this work by a modular form we mean a modular form with possibly non-trivial multiplier system.\\

The local expansion of a modular form $\phi$ of weight $k\in \mathbb{Z}$ near the infinity cusp $\tau=i\infty$ is obtained by its Fourier development.
In a neighborhood of an interior point $\tau_{*}\in\mathbb{H}$, 
the natural local coordinate for the expansion is
the local uniformizing parameter based at the point $\tau_{*}$. It is defined by
\begin{eqnarray}
\mathbb{H}&\rightarrow& \mathbb{D}\,\nonumber\\
 \tau&\mapsto& S(\tau;\tau_{*})={\tau-\tau_{*}\over \tau-\bar{\tau}_{*}}\,,
\end{eqnarray}
which maps the upper-half plane $\mathbb{H}$ to the unit disk $\mathbb{D}$.
The most natural way \cite{Zagier:2008} to expand the modular form $\phi$ is to consider
\begin{equation}
(\tau(S)-\bar{\tau}_{*})^{k}\phi (\tau(S))=(\tau_{*}-\bar{\tau}_{*})^{k}(1-S)^{-k} \phi (\tau(S))
\end{equation}
as a series in the local uniformizing parameter $S$. Zagier \cite{Zagier:2008} also gives a nice formula for the series whose coefficients are determined by
the derivatives of $\phi (\tau)$ evaluated at $\tau_{*}$.
Furthermore, the coefficients can be computed recursively by making use of the ring structure \cite{Kaneko:1995} of quasi-modular forms and 
are often "algebraic numbers containing interesting arithmetic information".

The above definition for modular forms can be generalized to almost-holomorphic modular forms in $\widehat{M}(\Gamma)$ directly (but not to quasi-modular forms in $\widetilde{M}(\Gamma)$) for a modular group $\Gamma<\mathrm{SL}_{2}(\mathbb{Z})$ which are studied in \cite{Kaneko:1995} and will be reviewed in 
Section \ref{secinvarianceRamanujan} below.
According to the structure theorem in \cite{Kaneko:1995}, 
we think of 
an almost-holomorphic modular form $\phi$
as a real analytic function in the coordinates $(\tau,\bar{\tau})$.

For accuracy, we introduce the following definition.

\begin{dfn}[Local expansion of an almost-holomorphic modular form]\label{dfnlocalexpansion}
Suppose $\phi(\tau,\bar{\tau})\in \widehat{M}(\Gamma)$ is an almost-holomorphic modular form of weight $k\in \mathbb{Z}$. 
Consider the \emph{Cayley transform} with the normalization given by
\begin{equation}\label{eqnlocaluniformizingvariable}
\mathcal{C}: \tau\mapsto s(\tau;\tau_{*})={\tau-\tau_{*}\over {\tau\over \tau_{*}-\bar{\tau}_{*} }-  {\bar{\tau}_{*}\over \tau_{*}-\bar{\tau}_{*} }}\,,
\end{equation}
corresponding to a transform $\gamma_{s}\in \mathrm{SL}_{2}(\mathbb{C})$.
The local expansion of $\phi$ near $\tau_{*}$ is defined to be the Laurent series of the function, which we call the \emph{Cayley transformation of $\phi$}, 
\begin{equation}\label{eqnexpansion}
\mathscr{C}(\phi): (s,\bar{s})\mapsto \left(j_{\gamma_{s}}(\tau(s))\right)^{k}\phi(\tau(s,\bar{s}),\bar{\tau}(s,\bar{s}))\,.
\end{equation}
The prefactor $j_{\gamma}(\tau)$ is the $j$-automorphy factor defined by
\begin{equation}\label{automorphyfactor}
j: \mathrm{SL}_{2}(\mathbb{C})\times \mathbb{H}\rightarrow \mathbb{C}\,,\quad (\gamma=\begin{pmatrix} a&b\\c&d\end{pmatrix}, \tau)\mapsto j_{\gamma}(\tau)=(c\tau+d)\,.
\end{equation}
\end{dfn}
Henceforward for simplicity we shall denote $j_{\gamma_{s}}(\tau(s))$ by $j_{s}$, as a function of $s$ it is
\begin{equation*}
j_{s}=(1-{s\over \tau_{*}-\bar{\tau}_{*}})^{-1}\,.
\end{equation*}
Here in order to avoid the complexification of having to deal with the factor $\det(\gamma)$ in the usual definition of the $j$-automorphy factor,  we have normalized the entries in the Cayley transform $\mathcal{C}$ so that the resulting transform $\gamma_{s}$ gives an element in $\mathrm{SL}_{2}(\mathbb{C})$. 
Later this normalization will be naturally obtained by using monodromy calculation of the differential equations satisfied by the modular forms appearing in Tab.   \ref{tablequasimodularity}. \\

Before we proceed, a few useful remarks are in order. 
One could interpret the prefactor $j_{s}(\tau(s))^{k}$ in \eqref{eqnexpansion} as the "Jacobian"
of the Cayley transform $\mathcal{C}$. It takes case of the analyticity of the change of coordinate
and hence should be combined with the analyticity  of the almost-holomorphic modular form $\phi(\tau,\bar{\tau})$. 
Indeed, we invoke the description of an almost-holomorphic modular form $\phi$ of modular weight $k$ for the modular group $\Gamma$ as a real analytic section $\phi(\tau,\bar{\tau})(d\tau)^{k\over 2}$ of the line bundle $\mathcal{K}^{k\over 2}$ over the modular curve $\Gamma\backslash\mathbb{H}^{*}$, where $\mathcal{K}$ is the canonical bundle of the modular curve.
The "transition function" arising from the change of coordinate, which is also the "Jacobian", is exactly the prefactor. 
That is,
\begin{equation*}
\phi(\tau,\bar{\tau})(d\tau)^{k\over 2}=\mathscr{C}(\phi)(s,\bar{s})(ds)^{k\over 2}\,,
\end{equation*}
with
\begin{equation*}
 \mathscr{C}(\phi)(s,\bar{s})=(j_{s}(\tau(s)))^{k}\phi(\tau(s,\bar{s}),\bar{\tau}(s,\bar{s}))\,.
\end{equation*}
The power $k$ in the prefactor indicates the representation which defines the homogeneous line bundle $\mathcal{K}^{{k\over 2}}$.
The description as a section is not valid for a quasi-modular form, whose local expansion will be considered later via the structure theorem \cite{Kaneko:1995} between the ring of quasi-modular forms and the ring of almost-holomorphic modular forms.

\subsection{Invariance of Ramanujan identities}
\label{secinvarianceRamanujan}

As an immediate consequence of Definition \ref{dfnlocalexpansion}, we now show that the Ramanujan identities satisfied by the generators of the ring of almost-holomorphic modular forms are invariant under the Cayley transformation.\\

For later use we now recall the fundamental structure theorem in \cite{Kaneko:1995} between the rings of quasi-modular forms $\widetilde{M}(\Gamma)$ and almost-holomorphic modular forms $\widehat{M}(\Gamma)$ for a modular group $\Gamma<\mathrm{SL}_{2}(\mathbb{Z})$.
The latter has the structure given by $\widehat{M}(\Gamma)\cong \widetilde{M}(\Gamma)\otimes \mathbb{C}[Y(\tau,\bar{\tau})]$, where
\begin{equation}
 Y(\tau,\bar{\tau}):={1\over \mathrm{Im}\tau}\,.
\end{equation}
One can regard an almost-holomorphic modular form as a polynomial in the non-holomorphic quantity $Y$ (with coefficients being holomorphic quantities). Taking the degree zero term of an almost-holomorphic modular form gives a quasi-modular form. On the contrary, if one starts with a quasi-modular form, one can complete it to an almost-holomorphic modular form by adding polynomials in $Y$. The former map is called the \emph{constant term map} and the latter \emph{modular completion}.
These two maps give the isomorphism between the two rings.

Moreover, if one starts with the generators of the ring of quasi-modular forms, then since the ring is closed under 
the derivative $\partial_{\tau}:={1\over 2\pi i}{\partial \over \partial \tau}$, one gets different equations among these generators called \emph{Ramanujan identities}.
Now one replaces the generators by their modular completions, and the derivative $\partial_{\tau}$ by the raising operator 
\begin{equation}\label{eqncovariantderivative}
\widehat{\partial}_{\tau}:=\partial_{\tau}+{k\over 12} {-3\over \pi}Y(\tau,\bar{\tau})
\end{equation}
when acting on a quasi-modular form of weight $k$. 
This operator is the natural covariant derivative, corresponding to the Chern connection, on sections of $\mathcal{K}^{{k\over 2}}$
with respect to the trivialization $(d\tau)^{k\over 2}$.
It turns out that one gets the same system of equations, now satisfied by the generators of the ring of almost-holomorphic modular forms. \\

The above results will be used frequently throughout this work, hence we summarize them in the following diagram
\begin{figure}[h]
  \renewcommand{\arraystretch}{1.2} 
\begin{displaymath}
\xymatrixcolsep{5pc}\xymatrix{  \partial_{\tau}\curvearrowright \widetilde{M}(\Gamma) \ar@/^/[r]^{\textrm{modular completion}}&  \ar@/^/[l]^{\textrm{constant term map}}  \widehat{M}(\Gamma) \curvearrowleft  \widehat{\partial}_{\tau} }
\end{displaymath}
  \caption[Structure theorem]{Structure theorem between the rings of quasi-modular forms and almost-holomorphic modular forms}
  \label{figurestructuretheorem}
\end{figure}

We now study the differential equations among the generators in the ring $\mathscr{C}(\widehat{M}(\Gamma))$.
Recall the $j$-automorphy factor $j_{s}$ defined in \eqref{automorphyfactor}. We define the raising operator $\widehat{\partial}_{s}$ such that for an almost-holomorphic modular form $\phi$ of weight $k$ 
\begin{equation}
\widehat{\partial}_{s} \mathscr{C}(\phi)=
 \mathscr{C}(\widehat{\partial}_{\tau}\phi)
 \,.
\end{equation}
By construction, the right hand side is $j_{s}^{k+2}\widehat{\partial}_{\tau}\phi$.
Straightforward calculation shows that on the Cayley transformation $\mathscr{C}(\phi)$ of an almost-holomorphic modular form $\phi$  of weight $k$, one has
\begin{equation}
\widehat{\partial}_{s}=\partial_{s}+{k\over 12} (j_{s}^{2} {-3\over \pi \,\mathrm{Im}\tau} - 12 j_{s}\partial_{\tau}j_{s})
=\partial_{s}+{k\over 12}  {-3\over \pi }(j_{s}^{2} {1\over  \mathrm{Im}\tau} -2 \sqrt{-1} j_{s}{\partial \over \partial \tau}j_{s})\,.
\end{equation}

Similar to \eqref{eqncovariantderivative}, this operator is the covariant derivative on the same line bundle $\mathcal{K}^{{k\over 2}}$, but
with respect to the trivialization $(ds)^{k\over 2}$.
Note that due to the fact that $\gamma_{s}$ is not an element in $\mathrm{SL}_{2}(\mathbb{R})$, the quantity $(j_{s}^{2} {1\over  \mathrm{Im}\tau} -2 \sqrt{-1} j_{s}{\partial \over \partial \tau}j_{s})$ is not equal to $1/\mathrm{Im}\, s$. In fact, the above extra term one adds to $\partial_{s}$ is a multiple of
\begin{equation}\label{eqnYC}
\mathscr{C}(Y)(s,\bar{s}):=(j_{s}^{2} {1\over  \mathrm{Im}\tau} -2  \sqrt{-1} j_{s}{\partial \over \partial \tau}j_{s})
=2  \sqrt{-1} {{\bar{s}\over K^{2}} \over {1+{s\bar{s}\over K^{2}}}}\,, \quad 
 K=\bar{\tau}_{*}-\tau_{*}\,.
\end{equation}
Now by construction, we have the following proposition.
\begin{prop}[Invariance of Ramanujan identities for almost-holomorphic modular forms]\label{thminvariancealmostmodular}
Consider the differential relations 
\begin{equation}
\widehat{\partial_{\tau}}\phi_{a}=P_{a}(\phi_{1},\cdots)\,,\quad a=1,2,\cdots \,,
\end{equation}
where $\{\phi_{a}\,,\, a=1,2\cdots\}$ give the generators for the ring of almost-holomorphic modular forms and $P_{a}$ are polynomials giving the Ramanujan identities. Then these differential relations are satisfied, under the Cayley transformation, by the local expansions $\mathscr{C}(\phi_{a})$ of $\phi_{a}\,,\,a=1,2,\cdots $. That is, one has
\begin{equation}
\widehat{\partial_{s}} \,\mathscr{C}(\phi_{a}) =P_{a}(\mathscr{C}(\phi_{1}),\cdots)\,, \quad a=1,2,\cdots \,.
\end{equation}
\end{prop}
This is of course nothing but a manifestation of the different descriptions of almost-holomorphic modular forms as functions on $\mathbb{H}$ or on $\mathbb{D}$ (with a rescaled radius due to the normalization of the Cayley transform we have taken).
The denominator in the quantity $\mathscr{C}(Y)$ should be regarded as the remnant of the hyperbolic metric \cite{Shimura:1987}.

\subsubsection{Holomorphic limit}

We shall need the notion of holomorphic limit for later discussions. 
This notion makes sense for all real analytic functions. See \cite{Bershadsky:1993cx, Kapranov:1999, Gerasimov:2004yx} and the more recent work \cite{Zhou:2013hpa} for details on this. 
In particular, for an almost-holomorphic modular form expanded near the infinity cusp, the holomorphic limit is given by the constant term map from
$\widehat{M}(\Gamma)$ to $\widetilde{M}(\Gamma)$
in the expansion of the almost-holomorphic modular form in terms of the formal variable $1/\mathrm{Im}\tau$.
This map can be thought of induced by the so-called holomorphic limit $\bar{\tau}\rightarrow \overline{i\infty}$. 
Around a point $\tau_{*}$ in the interior of the upper-half plane $\mathbb{H}$, one can also consider the holomorphic limit of the Cayley transformation $\mathscr{C}(\phi)$ of an almost-holomorphic modular form $\phi\in \widehat{M}(\Gamma)$, denoted by 
\begin{equation}
\mathrm{hlim}_{\bar{\tau}\rightarrow \bar{\tau}_{*}}\,\mathscr{C}(\phi)\,.
\end{equation}
The holomorphic limit at the point $\tau_{*}$ is induced by setting $\bar{\tau}\mapsto \bar{\tau}_{*}$, which is equivalent to the limit $\bar{s}\rightarrow 0$ according to \eqref{eqnlocaluniformizingvariable}.

By definition, the notion of holomorphic limit also depends on the base point $\tau_{*}$ and holomorphic limits of the same real analytic function at different points are in general not related by analytic continuation.
The quantity of particular importance in this work is the non-holomorphic function $Y=1/\mathrm{Im}\tau$. One has 
\begin{equation}
\mathrm{hlim}_{\bar{\tau}\rightarrow \overline{i\infty}}{1\over \mathrm{Im}\tau}=0\,,\quad \mathrm{hlim}_{\bar{\tau}\rightarrow \bar{\tau}_{*}}{1\over \mathrm{Im}\tau}={2\sqrt{-1}\over \tau-\bar{\tau}_{*}}\,.
\end{equation}
It is easy to see that
\begin{eqnarray}
\mathrm{hlim}_{\bar{\tau}\rightarrow \overline{i\infty}}\mathscr{C}(Y)(s,\bar{s})&=&\mathrm{hlim}_{\bar{s}\rightarrow  K}\mathscr{C}(Y)(s,\bar{s})=2\sqrt{-1}( {{1\over K}\over 1+{s\over K}})
\,,\nonumber\\
\mathrm{hlim}_{\bar{\tau}\rightarrow \bar{\tau}_{*}}\mathscr{C}(Y)(s,\bar{s})&=&\mathrm{hlim}_{\bar{s}\rightarrow 0}\mathscr{C}(Y)(s,\bar{s})=0\,.
\end{eqnarray}

Recall that the Ramanujan identities for the generators of the ring of quasi-modular forms (expanded around the infinity cup) can be obtained by applying the constant map on the differential equations for the generators of the ring of almost-holomorphic modular forms, as shown in Fig. \ref{figurestructuretheorem}. Computationally this follows from the structure theorem between these two rings and the fact that $\partial_{\tau}Y=cY^{2}$ for some constant $c$ and hence vanishes in the holomorphic limit.
More precisely, near the infinity cusp, the notions of almost-holomorphic modular forms and quasi-modular forms are such that the modular completion of a quasi-modular form $f$ of weight $k$ is given by an almost-holomorphic modular form as a polynomial in $Y$ \cite{Kaneko:1995}:
\begin{equation}\label{holomorphic-limit}
\phi(\tau,\bar{\tau})=f(\tau)+\sum_{m=1}^{[{k\over 2}]}f_{m}(\tau)Y^{m}\,.
\end{equation}

Note that $h=\mathscr{C}(Y)-j_{s}^{2}Y$ is a purely holomorphic quantity, we then define quasi-modular forms around $\tau_{*}$ to be the holomorphic limit at $\tau_{*}$
of the Cayley transformations of almost-holomorphic modular forms. Then we have a similar structure theorem that an almost-holomorphic modular form is the sum of its holomorphic limit and polynomials in $\mathscr{C}(Y)-h$. That is, the non-holomorphic part of $\mathscr{C}(\phi)$ is a polynomial (coefficients are holomorphic functions) in $\mathscr{C}(Y)-h$
or equivalently a polynomial in $\mathscr{C}(Y)$. The quantity $h$ measures the deviation between the two notions of holomorphic limits and hence of quasi-modular forms.
It is easy to see that the notion of quasi-modular form does not depend on the normalization for the Cayley transform since that of the holomorphic limit does not.
In fact, from \eqref{eqnYC} and the above definition of the holomorphic limit such that $\mathrm{hlim}_{\bar{\tau}\rightarrow \bar{\tau}_{*}}\mathscr{C}(Y)=0$, we can see that taking the holomorphic limit of 
the Cayley transformation $\mathscr{C}(\phi)=j_{s}^{k}\phi$ of the left hand side in \eqref{holomorphic-limit} amounts to replacing 
$j_{s}^2 Y=\mathscr{C}(Y)+h$ by $h=2  \sqrt{-1} j_{s}\partial_{\tau}j_{s}$ in the expansion of $j_{s}^{k}\phi$ on the right hand side.
This is nothing but formally (that is, as if $\tau\mapsto s$ lies in $\mathrm{SL}_{2}(\mathbb{Z})$) applying the transformation law to the quasi-modular form $f$.

Again the same computation as in the infinity cusp case says that one would get a system of differential equations satisfied by these holomorphic limits, with the ordinary derivative $\partial_{s}$. 
That is, although the notion of holomorphic limit varies when the base point moves, different holomorphic limits do obey the same differential equation system.
Hence we obtain the following proposition.

\begin{prop}[Invariance of Ramanujan identities for quasi-modular forms]\label{holomorphic-invariance}\label{thminvariancequasimodular}
Consider the differential relations 
\begin{equation}
\partial_{\tau}f_{a}=P_{a}(f_{1},\cdots) \,, \quad a=1,2,\cdots \,,
\end{equation}
where $\{f_{a}\,,\, a=1,2\cdots\}$ give the generators for the ring of quasi-modular forms and $P_{a}$ are polynomials giving the Ramanujan identities. 
Denote the modular completion of $f_{a}$ to be $\phi_{a}$. Then these differential relations are satisfied, under the Cayley transformation, by the holomorphic limits $\mathrm{hlim}\mathscr{C}(\phi_{a}) $ of the quantities $\mathscr{C}(\phi_{a})$, with the derivative replaced by $\partial_{s}$. That is, one has
\begin{equation}
\partial_{s}(\mathrm{hlim}\, \mathscr{C}(\phi_{a}) )=P_{a}(\mathrm{hlim}\,\mathscr{C}(\phi_{1}) \,,\cdots)\,, \quad a=1,2,\cdots \,.
\end{equation}
\end{prop}
Henceforward, we shall denote $\mathrm{hlim}\,\mathscr{C}(\phi)$ by $\mathscr{C}_{\mathrm{hol}}(f)$ if the non-holomorphic modular form $\phi$ is the modular completion at the infinity cusp of the quasi-modular form $f$.
If $f$ is a holomorphic modular form, then by construction $\mathscr{C}_{\mathrm{hol}}(f)=\mathscr{C}(f)$.

\begin{ex}
[Full modular group case] Consider the case $\Gamma=\mathrm{SL}_{2}(\mathbb{Z})$, then the ring of modular forms $M(\Gamma)$, quasi-modular forms $\widetilde{M}(\Gamma)$ and almost-holomorphic modular forms $\widehat{M}(\Gamma)$ are given by
\begin{equation*}
M(\Gamma)=\mathbb{C}[E_{4},E_{6}]\,,\quad 
\widetilde{M}(\Gamma)=\mathbb{C}[E_{4},E_{6}][E_{2}]\,,\quad \widehat{M}(\Gamma)=\mathbb{C}[E_{4},E_{6}][\hat{E}_{2}]\,,
\end{equation*}
where
\begin{equation*}
\hat{E}_{2}(\tau,\bar{\tau})=E_{2}(\tau)+{-3\over \pi}{1\over \mathrm{Im}\tau}\,.
\end{equation*}
Around an interior point $\tau_{*}$, we again choose the Cayley transform to be the map in \eqref{eqnlocaluniformizingvariable}.
Now the equivalent descriptions of the ring of modular forms and almost-holomorphic modular forms around this point are obtained via the Cayley transformation as 
\begin{eqnarray*}
\mathscr{C}(M(\Gamma))&=&\mathbb{C} [j_{s}^{4} E_{4}(\tau(s)),j_{s}^{6} E_{6}(\tau(s))]\,,\\
\mathscr{C}(\widehat{M}(\Gamma))&=&\mathbb{C}[j_{s}^{4}E_{4}(\tau(s)),j_{s}^{6} E_{6}(\tau(s))][j_{s}^{2} \hat{E}_{2}(\tau(s,\bar{s}),\bar{\tau}(s,\bar{s}))]\,.
\end{eqnarray*}
However, the ring generated by the holomorphic limits of the generators of the ring $\widehat{M}(\Gamma)$ is
\begin{equation*}
\mathscr{C}_{\mathrm{hol}}(\widetilde{M}(\Gamma))=\mathbb{C}[j_{s}^4 E_{4}(\tau(s)), j_{s}^{6} E_{6}(\tau(s))][
\mathscr{C}_{\mathrm{hol}}(E_{2})]\,,
\end{equation*}
where
\begin{equation*}
\mathscr{C}_{\mathrm{hol}}(E_{2})=\mathrm{hlim}_{\bar{s}\rightarrow 0} \left(j_{s}^{2}\hat{E}_{2}(\tau(s,\bar{s}) ,\bar{\tau}(s,\bar{s}))\right)
=j_{s}^{2}E_{2}(\tau(s,\bar{s}))+j_{s}^{2}{-3\over \pi}{2\sqrt{-1}\over \tau(s)-\bar{\tau}_{*}}\,.
\end{equation*}
The differential equations take the same form as the classical Ramanujan identities among the Eisenstein series $E_{2}(\tau),E_{4}(\tau),E_{6}(\tau)$.
\end{ex}

The results in Proposition \ref{thminvariancealmostmodular} and Proposition \ref{thminvariancequasimodular} can be summarized into the commutative diagram in Fig. \ref{figureinvarianceRamanujan}. Here as before the notations $\mathcal{O}, C^{\omega}$ mean the ring of holomorphic and real analytic functions, respectively.
In particular, the Cayley transformation preserves the differential ring structure.
Again this is a manifestation of the gauge invariance of the Chern connection
on the Hermitian line bundle $\mathcal{K}^{{k\over 2}}$.
\begin{figure}[h]
  \renewcommand{\arraystretch}{1.2} 
\begin{displaymath}
\xymatrixcolsep{5pc}\xymatrixrowsep{5pc}\xymatrix{  \partial_{\tau}\curvearrowright \widetilde{M}(\Gamma) \subseteq \mathcal{O}_{\mathbb{H}}
\ar@/^/[r]^{\textrm{modular completion}}\ar@{.>}[d]^{\mathscr{C}_{\mathrm{hol}}}   
 &  \ar@/^/[l]^{\textrm{constant term map}}  \widehat{M}(\Gamma)\subseteq  C^{\omega}_{\mathbb{H}} \ar[d]^{\mathscr{C}} \curvearrowleft  \widehat{\partial}_{\tau}\\
 \partial_{s}\curvearrowright \mathscr{C}_{\mathrm{hol}}(\widetilde{M}(\Gamma)) \subseteq \mathcal{O}_{\mathbb{D}}\ar@/^/[r]^{\textrm{modular completion}}&  \ar@/^/[l]^{\textrm{holomorphic limit}}  \mathscr{C}(\widehat{M}(\Gamma) )\subseteq  C^{\omega}_{\mathbb{D}}\curvearrowleft  \widehat{\partial}_{s}
 }
\end{displaymath}
  \caption[InvarianceofRamanujan]{Invariance of Ramanujan identities}
  \label{figureinvarianceRamanujan}
\end{figure}

The whole discussion above applies to a general modular group $\Gamma<\mathrm{SL}_{2}(\mathbb{Z})$.
In this work, we shall only consider the quasi-modular forms for $\Gamma_{0}(N),N=2,3,4$ which appear in Tab. \ref{tablequasimodularity}.\\

In principle, knowing: (a). around which point $\tau_{*}$ one should study the local expansions;
(b). the Ramanujan identities; and (c). some boundary conditions which can be calculated straightforwardly, are enough for the purpose of proving the LG/CY correspondence in Section \ref{secLGCYcorrespondence}.
However, in the following section we shall give a detailed study on the computations of the local expansions by taking advantage of the fact that the generators of the ring of modular forms are closely related to period integrals of certain elliptic curve families and hence satisfy simple differential equations. 
We hope this will be inspiring in understanding the general picture of the global properties of correlation functions in the enumerative theories since in most of the examples studied in the literature on LG/CY correspondence, it is usually unknown what the underlying modularity should be and what one can use are differential equations similar to what we have here. In the course we shall also explain the following important issues:
\begin{itemize}
\item To identity the base point $\tau_{*}$ needed for the purpose of matching the enumerative theories in consideration. This point is to be singled out as a singularity of the Picard-Fuchs equation
around which the monodromy action is diagonal in a particular basis of solutions. This singularity is called an orbifold singularity.

\item To motivate the definition of the normalized local uniformization variable \eqref{eqnlocaluniformizingvariable} from the perspective of  monodromy consideration of the Picard-Fuchs equation and also of the Weil-Petersson geometry on the moduli space.

\item To obtain close-form expressions for the Cayley transformations of the quasi-modular forms in terms of hypergeometric series.
These results offers an alternative way to match \cite{Krawitz:2011} the prepotentials of the FJRW theories with those of the simple elliptic singularities \cite{Noumi:1998}.

\end{itemize}

\subsection{Gauss hypergeometric equations for modular forms and their local expansions}
\label{secperiodcalculation}

We shall study the analytic continuations and elliptic expansions of the modular forms in consideration, with the help of the elliptic curve families described in Tab. \ref{tablequasimodularity}.
The modular forms are closely related to the periods of these elliptic curve families which satisfy differential equations known as the Picard-Fuchs equations.
These elliptic curve families can be thought of as computational devices or auxiliary objects in the sense that
the facts about modular forms that we are going to discuss are independent of the specific elliptic curve family that we have chosen. 

For simplicity, we take the$\mathbb{P}^{1}_{3,3,3}$ case for example. Similar discussions apply to the elliptic orbifold curves $\mathbb{P}^{1}_{4,4,2}, \mathbb{P}^{1}_{6,3,2}$. The pillowcase orbifold 
$\mathbb{P}^{1}_{2,2,2,2}$ will be related to the $\mathbb{P}^{1}_{4,4,2}$ case since according to the computations in \cite{Shen:2014} the correlators involved in the enumerative theory that 
we are interested in are actually quasi-modular forms for $\Gamma_{0}(2)$. \\

The elliptic curve family in this case is given by the Hesse pencil
\begin{equation}
x_{1}^{3}+x_{2}^{3}+x_{3}^{3}-3\alpha^{-{1\over 3}}x_{1}x_{2}x_{3}=0\,, \quad j(\alpha)={27(1+8\alpha)^{3}\over \alpha(1-\alpha)^{3}}\,,
\end{equation}
defined on the modular curve $\Gamma_{0}(3)\backslash \mathbb{H}^{*}$ which is parametrized by the Hauptmodul $\alpha$.
Hereafter we shall fix a particular normalization of $\alpha$ 
following \cite{Maier:2009} given by
\begin{equation}\label{eqnHauptmodul}
\alpha(\tau)={({3\eta^{3}(3\tau)\over \eta(\tau)})^{3} \over ({3\eta^{3}(3\tau)\over \eta(\tau)})^{3}+({\eta(\tau)^{3}\over \eta(3\tau)})^{3} }\,.
\end{equation}
The base of the family is an orbifold whose coarse moduli is $\mathbb{P}^{1}$, with three orbifold points of indices $\infty,\infty,3$.

For the modular curve $\Gamma_{0}(3)\backslash \mathbb{H}^{*}$, the cusps are 
$[\tau]=[i\infty],[0]$, the orbifold point is $[\tau]=[ST^{-1}(\zeta_{3})]=-\kappa \zeta_{3}$, where 
$S,T$ are the generators for the full modular group $\mathrm{SL}_{2}(\mathbb{Z})$ and
\begin{equation}
\kappa:={i\over \sqrt{3}}\,, \quad  \zeta_{3}=\exp({2\pi i  \over 3})\,.
\end{equation}
They correspond to singularities 
$\alpha=0,1,\infty$ of the elliptic curve family, respectively.  
The generators for the ring of modular forms can be nicely constructed from the triple 
$A_{3},B_{3},C_{3}$ in \cite {Borwein:1991, Borwein:1994, Maier:2009}.
To make the paper self-contained, we recall these quantities
\begin{equation}
A_{3}(\tau)=\theta_{2}(2\tau)\theta_{2}(6\tau)+\theta_{3}(2\tau)\theta_{3}(6\tau)\,, \quad B_{3}(\tau)={3\eta(3\tau)^{3}\over \eta(\tau)}\,,\quad  C_{3}(\tau)={\eta(\tau)^{3}\over \eta(3\tau)}\,.
\end{equation}
Here we have chosen the convention in \cite{Zagier:2008} for the $\theta$--constants.
They satisfy the relation
\begin{equation}
A_{3}(\tau)^{3}=B_{3}(\tau)^{3}+C_{3}(\tau)^{3}\,.
\end{equation}
We also choose the following quantity as a generator for the ring of quasi-modular forms,
\begin{equation}
E_{3}(\tau):={3E_{2}(3\tau)+E_{2}(\tau)\over 3}\,.
\end{equation}
This generator $E_{3}$ can be completed to an almost-holomorphic modular form
\begin{equation}
\hat{E}_{3}(\tau,\bar{\tau})=E_{3}(\tau)+{1\over 2}{-3\over \pi }{1\over \mathrm{Im} \tau}\,.
\end{equation}
These are the modular forms whose global properties on the moduli space and whose local expansions in different regions in the moduli space that we want to study in this work.

Note that the equivalence class of the elliptic point $[\tau]$ has many representatives which are related by transformations in $\Gamma_{0}(3)$, the representative
$ST^{-1}(\zeta_{3})$ is singled out according to the normalization of the Hauptmodul we have fixed.
This can also be seen later by using analytic continuation of periods, see Remark \ref{remDeck}.

The explicit Ramanujan identities satisfied by the generators $A_{3},B_{3},C_{3},E_{3}$ are given in \eqref{ramanujan} by \cite{Zhou:2013hpa} (which differs from the ones used in \cite{Maier:2009} due to the different choice for $E_{3}$ and makes the equations more systematic). 

\subsubsection{Local uniformizing parameter near elliptic point obtained from monodromy}

The Picard-Fuchs operator for the Hesse pencil is
\begin{equation}\label{eqPFaroundinfinity}
\mathcal{L}=\theta_{\alpha}^{2}-\alpha(\theta_{\alpha}+{1\over 3})(\theta_{\alpha}+{2\over 3})\,, \quad \theta_{\alpha}:=\alpha{\partial \over \partial \alpha}\,.
\end{equation}
This is the Gauss hypergeometric differential operator with $a=1/3, b=2/3,c=1$.
The corresponding Picard-Fuchs equation has three singularities located at
$\alpha=0,1,\infty$.

The triple $A_{3}(\tau),B_{3}(\tau),C_{3}(\tau)$ are related to the periods of this elliptic curve family in the following way \cite{Borwein:1991, Borwein:1994, Maier:2009}
\begin{equation}
A_{3}(\tau)=F(\alpha(\tau))\,,\quad B_{3}(\tau)=(1-\alpha(\tau))^{1\over 3}F(\alpha(\tau))\,,\quad C_{3}(\tau)=\alpha(\tau)^{1\over 3}F(\alpha(\tau))\,,
\end{equation}
where
$F(\alpha)=\,_{2}F_{1}(a,b;c;\alpha)$ is a solution to the Picard-Fuchs equation and $\alpha(\tau)$ is given by \eqref{eqnHauptmodul}.
The quasi-modular form $E_{3}(\tau)$ can be related to the periods according to
\begin{equation}
\partial_{\tau}\log C_{3}(\tau)={1\over 6}(E_{3}(\tau)+A_{3}(\tau)^{2})\,,
\end{equation}
The Schwarz relation is also useful
\begin{equation}
\partial_{\tau}\alpha(\tau)=\alpha(\tau)(1-\alpha(\tau))F(\alpha(\tau))^{2}\,.
\end{equation}
Similar results hold for the other modular groups $\Gamma_{0}(N),N=2,4$.
We refer the interested readers to 
\cite{Maier:2009, Maier:2011, Zhou:2013hpa}
for further details.

According to the discussion in Appendix \ref{appendixB},
consideration of monodromy leads to the following natural local uniformizing parameter \eqref{eqlocaluniformizingellipticpoint} near the orbifold singularity $\alpha=\infty$
\begin{equation}\label{eqnflatcoordinate}
\tau_{\mathrm{orb}}= {\tau-\tau_{*}\over { \tau\over \tau_{*}-\bar{\tau}_{*}}-{\bar{\tau}_{*}\over \tau_{*}-\bar{\tau}_{*}} }\,,\quad \tau_{*}=-{i\over \sqrt{3}}\exp({2\pi i\over 3})\,.
\end{equation}
Here the normalization is fixed by requiring that the transformation $\tau\mapsto \tau_{\mathrm{orb}}$ belongs to $\mathrm{SL}_{2}(\mathbb{R})=\mathrm{Sp}_{1}(\mathbb{R})$. It comes from the intuition that the action of the connection matrix on the correlation functions discussed later should be induced by a change of polarization in the context of Givental's formalism \cite{Givental:20011}. 
This local coordinate is exactly the normalized Cayley transform $\mathcal{C}(\tau)$ given in \eqref{eqnlocaluniformizingvariable}.
\begin{rem}
Similarly, one can define a local coordinate around any point $\tau_{*}$ according to the first equation in \eqref{eqnflatcoordinate}.
This coordinate coincides with the flat coordinate in \cite{Bershadsky:1993cx, Gerasimov:2004yx} by considering deformations in the Kodaira-Spencer theory, and the K\"ahler normal coordinate $t_{\mathrm{metric}}$ with respect to the Poincar\'e metric or equivalently Weil-Petersson metric on the upper-half plane, as computed in \cite{Zhou:2013hpa}.
See also \cite{Kapranov:1999} for a nice account of these discussions.
In particular, it is well defined around the infinity cusp $\tau_{*}=i\infty$. 
The normalization in this context is chosen such that to first order, the K\"ahler metric in the K\"ahler normal coordinate is represented by the identity matrix.
For this reason,  $t_{\mathrm{metric}}$ is sometimes called the flat coordinate.

Therefore, the flat coordinate/K\"ahler normal coordinate $t_{\mathrm{metric}}$, the
normalized period $\tau_{\mathrm{orb}}$ and the normalized Cayley transform $\mathcal{C}(\tau)$ are all the same.
This explains why the enumerative expansion, which usually appears in the expansion according to the flat coordinate/K\"ahler normal coordinate, naturally singles out the normalized local uniformizing variable/Cayley transform.
\end{rem}
\begin{rem}\label{remDeck}
According to the calculation in Appendix \ref{appendixB}, at the orbifold point $\tau_{\mathrm{orb}}=0$, we have
$\tau=-\kappa \zeta_{3}$.
Hence indeed the elliptic point is presented by $\tau=ST^{-1}(\zeta_{3})$.
This is why in the coordinate $\tau$, the singular points, as $\Gamma_{0}(3)$-equivalence classes, are represented by
$\tau=i\infty, 0, -\kappa \zeta_{3}$.

If one insists that the representative of the $\Gamma_{0}(3)$-elliptic point is, say, $\zeta_{3}$, then one needs to
apply the Deck transformation: $\tilde{\tau}=(ST^{-1})^{-1}\tau$. Hence in the $\tilde{\tau}$ coordinate the points $\alpha=0, 1, \infty$
are not given by $\tilde{\tau}=i\infty, 0, ST^{-1}(\zeta_{3})$ anymore, but by
$\tilde{\tau}=1, i\infty, \zeta_{3}$.
This normalization is not as convenient as the previous one since in the physics motivation, see e.g., 
\cite{Witten:1993, Seiberg:1994rs},
the point $\alpha=0$
which represents the large complex structure limit
should be in the weak coupling region $\tau\sim i\infty$.
\end{rem}

\subsubsection{Local expansions around the elliptic point}

The local expansion near the cusp $[\tau]=[0]$ is of interest in some other context \cite{Alim:2013eja} and will not be discussed here. We now discuss the expansion near the elliptic point.

We first consider the modular form $A_{3}(\tau)$
which is represented through the hypergeometric series $\pi_{1}(\alpha)=\,_{2}F_{1}(a,b;c;\alpha)$ near $\alpha=0$.
Here and in the rest of the work, we shall use the convention in \cite{Erdelyi:1981} for the solutions $\pi_{1},\pi_{2},\cdots$ to the hypergeometric equation.
The detailed expressions for them are recalled in Appendix \ref{appendixB}.

According to Definition \ref{dfnlocalexpansion}, the local expansion of near $\tau_{*}=-\kappa \zeta_{3}$ or equivalently $\tau_{\mathrm{orb}}=0$ is
\begin{equation*}
\mathscr{C}(A_{3})(\tau_{\mathrm{orb}})=({1\over 1+K^{-1}\tau_{\mathrm{orb}}})\pi_{1}(\alpha (\tau_{\mathrm{orb}}) )\,,\quad  K=\bar{\tau}_{*}-\tau_{*}\,.
\end{equation*}
Using the results on analytic continuation of periods in Appendix \ref{appendixB}, we have
\begin{equation*}
\pi_{1}(\alpha(\tau_{\mathrm{orb}}))=\tilde{\pi}_{3}(\alpha(\tau_{\mathrm{orb}}))+\tilde{\pi}_{4}(\alpha(\tau_{\mathrm{orb}}))
=\tilde{\pi}_{3}(\alpha(\tau_{\mathrm{orb}})) (1+K^{-1}\tau_{\mathrm{orb}})\,.
\end{equation*}
It follows that
\begin{equation}
\mathscr{C}(A_{3})(\tau_{\mathrm{orb}})=\tilde{\pi}_{3}(\alpha(\tau_{\mathrm{orb}}))\,.
\end{equation}
The quantity $\tilde{\pi}_{3}$ is most conveniently written as a hypergeometric series, up to the factor $(-\alpha)^{-{1\over 3}}$, in
$
\psi:=\alpha^{-{1\over 3}}.
$
Therefore, for computational purpose, we only need to find $\psi(\tau_{\mathrm{orb}})$.
Rewriting the Picard-Fuchs equation as
\begin{equation}\label{eqPFnearorbifold}
\left((\theta_{\psi}-1)(\theta_{\psi}-2)-\psi^{3} \theta_{\psi}^{2}\right) \tilde{\pi}_{3},\tilde{\pi}_{4}=0\,, \quad \theta_{\psi}=\psi {\partial \over \partial \psi}\,,
\end{equation}
we can see that the ratio of the two solutions to this differential equation, namely the coordinate $\tau_{\mathrm{orb}}$, satisfies the Schwarzian
\begin{equation}
\{\tau_{\mathrm{orb}},\psi\}:={\partial_{\psi}^{3} \tau_{\mathrm{orb}}\over \partial_{\psi} \tau_{\mathrm{orb}}}-{3\over 2}\left({\partial_{\psi}^{2} \tau_{\mathrm{orb}} \over \partial_{\psi} \tau_{\mathrm{orb}}}\right)^{2}=2Q(\psi)\,.
\end{equation} 
Here the $Q$-value $Q(\psi)$ is the rational function 
\begin{equation*}
\quad Q(\psi)=
{\psi (8+\psi^{3}) \over 2(1-\psi^{3})^{2}}\,.
\end{equation*}
Now using the property
\begin{equation}\label{eqSchwarzian}
\{\psi, \tau_{\mathrm{orb}}\}=-(\partial_{\tau_{\mathrm{orb}}}\psi)^{2} \{\tau_{\mathrm{orb}},\psi \}=-(\partial_{\tau_{\mathrm{orb}}}\psi)^{2} Q(\psi)\,,
\end{equation}
one gets a differential equation satisfied by $\psi(\tau_{\mathrm{orb}})$.
With the obvious boundary conditions, one can 
solve for $\psi(\tau_{\mathrm{orb}})$.
Plugging this into $\tilde{\pi}_{3}(\psi(\tau_{\mathrm{orb}}))$, one obtains the local expansion of $\mathscr{C}(A_{3})$.
Similarly, one can get the local expansion of the Hauptmodul $\alpha$ by using $\mathscr{C}(\alpha)(\tau_{\mathrm{orb}})=(\psi(\tau_{\mathrm{orb}}))^{-3}$.
In this way, we can obtain the local expansions of the modular forms
$A_{3},C_{3}=\alpha^{1\over 3}A_{3}$ near the elliptic point $\tau_{*}$.
\begin{rem}
If one works out the Schwarzian equations satisfied by $\tau(\alpha)$ and $\tau_{\mathrm{orb}}(\alpha)$, one can see that the $Q$-values are the same as they should be so since by construction $\tau,\tau_{\mathrm{orb}}$ are related by a factional linear transform and hence satisfy the same Schwarzian equation. 
The Schwarzian equations imply the following identities
\begin{equation}\label{eqnGSrelation}
\partial_{\tau}\alpha=\alpha(1-\alpha) \pi_{1}^{2}\,,
\end{equation}
\begin{equation}\label{eqnSchwarzian}
\partial_{\tau_{\mathrm{orb}}}\psi=c_{0} \psi (1-27\psi^{-3})\tilde{\pi}_{3}^{2}\,.
\end{equation}
Here $c_{0}=-81$ could be fixed by using boundary conditions.
Again it is easy to see that these two equations are related by a factional linear transform and are equivalent.  
Now instead of plugging the transcendental series $\psi(\tau_{\mathrm{orb}})$ solved from \eqref{eqSchwarzian} into the transcendental series $\tilde{\pi}_{3}(\psi)$ solved from \eqref{eqPFnearorbifold} , we could obtain
$\tilde{\pi}_{3}(\psi(\tau_{\mathrm{orb}}))$ by using the series  $\psi(\tau_{\mathrm{orb}})$ solved from \eqref{eqSchwarzian} and then plugging it into \eqref{eqnSchwarzian}. 

Another way to obtain the series expansion of $\tilde{\pi}_{3}(\tau_{\mathrm{orb}})$ is to derive an equation satisfied by $\tilde{\pi}_{3}(\psi(\tau_{\mathrm{orb}}))$. 
Straightforward calculation using \eqref{eqPFaroundinfinity}, \eqref{eqnGSrelation} gives the following ODE satisfied by $A_{3}$:
\begin{equation*}
-8(A_{3}')^{3}A_{3}'''+2A_{3} A_{3}' A_{3}'' A_{3}'''-(A_{3}')^2 (9A_{3}''A_{3}''+2A_{3} A_{3}'''')
+A_{3}^2(-A_{3}'''A_{3}'''+A_{3}''A_{3}'''')=0\,.
\end{equation*}
Here $'$ means the derivative $\partial_{\tau}={1\over 2\pi i}{\partial \over \partial \tau}$.
Since the equations \eqref{eqPFaroundinfinity}, \eqref{eqnGSrelation} are invariant under the fractional linear transform $\tau\mapsto \tau_{\mathrm{orb}}$ due to the definition of local expansions and the invariance of Ramanujan identities, the above equation for $A_{3}(\tau)$ also gives an equation satisfied by $\mathscr{C}(A_{3})(\tau_{\mathrm{orb}})$.
\end{rem}

We now consider the local expansions of quasi-modular forms, more specifically the quasi-modular form $E_{3}$, around the elliptic point.
Recall that the local expansions of $\mathscr{C}(A_{3}),\mathscr{C}(B_{3}),\mathscr{C}(C_{3})$ and $\mathscr{C}_{\mathrm{hol}}(E_{3})$ satisfy an ODE system 
which takes the form in \eqref{ramanujan} according to the structure theorem illustrated in Fig. \ref{figureinvarianceRamanujan}, 
with the derivative $\partial_{\tau}$ replaced by $\partial_{\tau_{\mathrm{orb}}}$. 
Hence the local expansion $\mathscr{C}_{\mathrm{hol}}(E_{3})$ is determined from
\begin{equation}
\partial_{\tau_{\mathrm{orb}}} \mathscr{C}(C_{3}) =
{1\over 6}\mathscr{C}(C_{3})(\mathscr{C}_{\mathrm{hol}}(E_{3})+\mathscr{C}(A_{3})^2)\,.
\end{equation}

\subsubsection{Change of normalization and rational local expansions}

Now we compute the explicit expansions $\mathscr{C}(A_{3}),\mathscr{C}(C_{3}),\mathscr{C}_{\mathrm{hol}}(E_{3})$
in the flat coordinate/K\"ahler normal coordinate/local uniformizing variable/Cayley transform $\tau_{\mathrm{orb}}$ given in \eqref{eqnormalizedperiodellipticpoint} in terms of the normalized period
\begin{equation}
\tau_{\mathrm{orb}}=K {\gamma_{-}\over \gamma_{+}}{\pi_{4}\over \pi_{3}}\,,\quad \gamma_{+}={\Gamma({1\over 3})\over\Gamma({2\over 3})^2}\,,\quad
\gamma_{-}={\Gamma(-{1\over 3})\over\Gamma({1\over 3})^2}\,,
\quad
K=\bar{\tau}_{*}-\tau_{*}\,.
\end{equation}
The first few terms of the results (can be obtained by using Mathematica for example) are
\begin{eqnarray*}
\mathscr{C}(A_{3})&=&\gamma_{+} \pi_{3}=\gamma_{+} (-\alpha)^{-{1\over 3}}\,_{2}F_{1}({1\over 3},{1\over 3};{2\over 3};\alpha^{-1})\,,\\
&=&\gamma_{+}({\gamma_{+}\over K \gamma_{-}}\tau_{\mathrm{orb}})\left(1+\mathcal{O}(({\gamma_{+}\over K \gamma_{-}}\tau_{\mathrm{orb}})^{6})\right)\,,\\
\mathscr{C}(C_{3})&=& \alpha^{1\over 3}\mathscr{C}(A_{3})=\gamma_{+}e^{-\pi i a}\,_{2}F_{1}({1\over 3},{1\over 3};{2\over 3};\alpha^{-1})\,\\
&=&\gamma_{+} e^{-\pi i a}\left(1-{1\over 6} ({\gamma_{+}\over K \gamma_{-}}\tau_{\mathrm{orb}})^{3}+\mathcal{O}(({\gamma_{+}\over K \gamma_{-}}\tau_{\mathrm{orb}})^{6})\right)\,,\\
\mathscr{C}_{\mathrm{hol}}(E_{3})&=& 6\,\partial_{\tau_{\mathrm{orb}}}\log \mathscr{C}(C_{3})-\mathscr{C}(A_{3}^2)\,,\\
&=&\gamma_{+}^{2} ({\gamma_{+}\over K \gamma_{-}}\tau_{\mathrm{orb}})^5 \left({1\over 10}+\mathcal{O}
(({\gamma_{+}\over K \gamma_{-}}\tau_{\mathrm{orb}})^6)\right)\,.
\end{eqnarray*}
Here we can also use \eqref{eqnSchwarzian} to simplify the last expression:
\begin{equation*}\label{eqnCayleyCintermsofperiods}
\mathscr{C}_{\mathrm{hol}}(E_{3})= 6{\partial_{\alpha^{-1}}\log \mathscr{C}(C_{3})\over \partial_{\alpha^{-1}}\tau_{\mathrm{orb}}}-\mathscr{C}(A_{3}^2)=-6 (\alpha^{-1}-1)\mathscr{C}(A_{3}^2)\partial_{\alpha^{-1}}\log \mathscr{C}(C_{3})-
\mathscr{C}(A_{3}^2) \,.
\end{equation*}
The coefficients of the local expansions involve Gamma-values. Up to the prefactors, these local expansions are actually series in the parameter ${\gamma_{+}\over K \gamma_{-}}\tau_{\mathrm{orb}}$ with rational coefficients.

For the later purpose of matching the rational enumerative invariants with the local expansions, it is natural to apply the following type of change of variable

\begin{equation}\label{eqchangofnormalization}
\tau_{\mathrm{orb}}\mapsto M\tau_{\mathrm{orb}}={M^{1\over 2}\tau_{\mathrm{orb}} \over M^{-{1\over 2}}}\,,
\end{equation}
 inducing the transformation on the local expansion
\begin{equation}\label{eqCayleyfromnormalization}
\mathscr{C}(\phi)\mapsto (M^{-{1\over 2}})^{k} \mathscr{C}(\phi)\,.
\end{equation}
Hence to cancel the prefactors in the local expansions we are led to
\begin{equation} \label{eqrationalparameter}
\tau_{\mathrm{orb}}\mapsto \tilde{\tau}_{\mathrm{orb}}= M_{\mathrm{rational}}\tau_{\mathrm{orb}}\,, \quad  M_{\mathrm{rational}}^{-{1\over 2}}=(\gamma_{+}e^{- i \pi a})^{-1}\,.
\end{equation}
That is, 
\begin{equation}
\tau_{\mathrm{orb}}\mapsto \tilde{\tau}_{\mathrm{orb}}= M_{\mathrm{rational}}\tau_{\mathrm{orb}}=
e^{-2\pi i a} \gamma_{+}\gamma_{-} K {\pi_{4}\over \pi_{3}}
={1\over 2\pi i} (-3)e^{- 2\pi i a}{\pi_{4}\over \pi_{3}}
\,.
\end{equation}
This would also simultaneously give the following nice relation
\begin{equation*}
{\gamma_{+}\over K \gamma_{-}}\tau_{\mathrm{orb}}=-{1\over 3}  e^{2\pi i a}(2\pi i \tilde{\tau}_{\mathrm{orb}})\,.
\end{equation*}
Now in the Ramanujan identities, the derivative becomes
$\partial_{\tilde{\tau}_{\mathrm{orb}}}:={1\over 2\pi i}{\partial \over \partial \tilde{\tau}_{\mathrm{orb}}}$.
In the new local coordinate $\tilde{\tau}_{\mathrm{orb}}$, we get
\begin{eqnarray}\label{333-initial}
\mathscr{C}(A_{3})&=& \alpha^{-{1\over 3}}\,_{2}F_{1}({1\over 3},{1\over 3};{2\over 3};\alpha^{-1})\,,\nonumber\\
&=&
{1\over 3} (2\pi i \tilde{\tau}_{\mathrm{orb}})
\left(1+\mathcal{O}((2\pi i \tilde{\tau}_{\mathrm{orb}})^{6})\right)\,,\nonumber\\
\mathscr{C}(C_{3})&=& \,_{2}F_{1}({1\over 3},{1\over 3};{2\over 3};\alpha^{-1})\,\nonumber\\
&=&
1+{1\over 6\cdot 3^3} (2\pi i \tilde{\tau}_{\mathrm{orb}})^3+\mathcal{O}((2\pi i \tilde{\tau}_{\mathrm{orb}})^{6})\,
,\nonumber\\
\mathscr{C}_{\mathrm{hol}}(E_{3})
&=&(2\pi i\tilde{\tau}_{\mathrm{orb}})^{5} \left(-{1\over 10\cdot 3^5}+\mathcal{O}((2\pi i \tilde{\tau}_{\mathrm{orb}})^{6})\right)\,.
\end{eqnarray}

\subsubsection{Local expansions of quasi-modular forms for $\Gamma_{0}(2)$}

The natural auxiliary elliptic curve family attached to the pillowcase orbifold in
Tab. \ref{tablequasimodularity}
is parametrized by the modular curve $\Gamma_{0}(4)\backslash\mathbb{H}^{*}$.
However, since there is no elliptic point on this modular curve, the method discussed above using periods does not apply directly. Instead, according to the results in \cite{Shen:2014}, we regard the correlation functions in the GW theory of the pillowcase orbifold as quasi-modular forms
for the modular group $\Gamma_{0}(2)$. Then we can expand them around the elliptic point $\tau_{*} $ on the modular curve $\Gamma_{0}(2)\backslash\mathbb{H}^{*}$.

In this case the Picard-Fuchs equation for the periods is the Gauss hypergeometric equation with $a=1/4, b=3/4,c=1$.
The same computations as in the $\mathbb{P}^{1}_{3,3,3}$ case give the following data for the elliptic point on the modular curve
$\Gamma_{0}(2)\backslash\mathbb{H}^{*}$
\begin{equation}
\kappa={i\over \sqrt{2}}\,, \quad \tau_{*}=\kappa e^{-\pi i {1\over 4}}\,,\quad K=\bar{\tau}_{*}-\tau_{*}\,.
\end{equation}
Similar calculation as in the $\Gamma_{0}(3)$ case yield expansions like
\begin{equation}
\mathscr{C}(C_{2})=\gamma_{+}e^{-\pi i a} \left(1+{1\over 8}({\gamma_{+}\over \gamma_{-}}\tau_{\mathrm{orb}})^2+
\mathcal{O}((\tau_{\mathrm{orb}})^4)\right)\,.
\end{equation}
Again to get rational expansions, we use the new coordinate 
\begin{equation}
\bar{\tau}_{\mathrm{orb}}=M_{\mathrm{rational}}\tau_{\mathrm{orb}}=
(\gamma_{+}e^{-\pi i a})^{2}\tau_{\mathrm{orb}}=\gamma_{+}\gamma_{-} e^{-2\pi i a} {\pi_{4}\over \pi_{3}}\,.
\end{equation}
In the new coordinate $\tilde{\tau}_{\mathrm{orb}}$, we have the following expansions
\begin{eqnarray}\label{eq2222boundary}
\mathscr{C}(A_{2})&=&(2\pi i \tilde{\tau}_{\mathrm{orb}})^{1\over 2}\left(2^{-{1\over 2}}+\mathcal{O}((2\pi i \tilde{\tau}_{\mathrm{orb}})^4)\right)\,,\nonumber\\
\mathscr{C}(C_{2})&=&1+{1\over 32}(2\pi i \tilde{\tau}_{\mathrm{orb}})^2+\mathcal{O}((2\pi i \tilde{\tau}_{\mathrm{orb}})^4)\,,\nonumber\\
\mathscr{C}_{\mathrm{hol}}(E_{2})&=&(2\pi i \tilde{\tau}_{\mathrm{orb}})^{3}\left(-{1\over 96}+\mathcal{O}((2\pi i \tilde{\tau}_{\mathrm{orb}})^4)\right)\,.
\end{eqnarray}

\section{LG/CY correspondence via modularity}
\label{secLGCYcorrespondence}

In this section, we use the results developed in earlier sections and in \cite{Krawitz:2011, Shen:2014} to give a proof of the LG/CY correspondence between the GW theories of elliptic orbifold curves and their FJRW counterparts.
We shall prove the following statement that the Cayley transformation $\mathscr{C}_{\rm hol}$ on quasi-modular forms induces the correspondence between the correlation functions in the GW and FJRW theories of a pair $(W,G)$.
\begin{thm}\label{lg-cy-main}
Let $(W,G)$ be a pair in Tab. \ref{tableWGdata}. Then there exists a degree and pairing preserving isomorphism between the graded vector spaces
\begin{equation}
\mathscr{G}: \left(\mathcal{H}^{\rm GW},\eta^{\rm GW}\right)\to \left(\mathcal{H}^{\rm FJRW},\eta^{\rm FJRW}\right)
\end{equation}
and a Cayley transformation $\mathscr{C}_{\rm hol}$, based at an elliptic point $\tau_{*}\in \mathbb{H}$,  such that for any $\{\alpha_j\}\subseteq \mathcal{H}^{\rm GW},$
\begin{equation}\label{lg-cy-correlator}
\mathscr{C}_{\rm hol}\left(\LL\alpha_1\psi_1^{\ell_1},\cdots,\alpha_k\psi_k^{\ell_k}\RR_{g,k}^{\rm GW}(q)\right)=
\LL\mathscr{G}(\alpha_1)\psi_1^{\ell_1},\cdots, \mathscr{G}(\alpha_k)\psi_k^{\ell_k}\RR^{\rm FJRW}_{g,k}(u)\,.
\end{equation}
\end{thm}

In this section, we shall give the details of the proof of Theorem \ref{lg-cy-main} for the elliptic orbifold curves $\mathbb{P}^1_{3,3,3}$ and $\mathbb{P}^1_{2,2,2,2}$  listed in Tab. \ref{tableWGdata} and Tab. \ref{tablequasimodularity}.
The elliptic point is given by the one on the modular curve $\Gamma_{0}(3)\backslash \mathbb{H}^{*}$ and $\Gamma_{0}(2)\backslash \mathbb{H}^{*}$, respectively. 
The other two cases can be proved similarly.

In each case, we first construct a map $\mathscr{G}$ preserving the degree and pairing of the state spaces. Then we check that the equality \eqref{lg-cy-correlator} holds for the building block correlation functions, which are polynomials in the generators that appear in the Ramanujan identities \eqref{ramanujan}. After that we check equality \eqref{lg-cy-correlator} holds for genus zero primary correlation functions, and in all genera by using the fact that the Cayley transformation $\mathscr{C}_{\rm hol}$ respects the WDVV equations and tautological relations which are used for the reconstruction of all genera correlation functions.

\begin{rem}
We make a remark about the relation between Theorem \ref{lg-cy-main} and the earlier works \cite{Krawitz:2011, Milanov:2011, Milanov:2012qu}.
We define $\mathcal{A}^{\rm FJRW}_{(W,G)}(u)$ by replacing the correlator $\langle\cdots\rangle$ in \eqref{fjrw-an} with correlation functions $\langle\langle\cdots\rangle\rangle(u)$, and define $\mathcal{A}^{\rm FJRW}_{(W,G)}(u, \bar{u})$ by replacing the correlation functions with their modular completions.
Similarly, we define $\mathcal{A}^{\rm GW}_{(W,G)}(\tau)$ and $\mathcal{A}^{\rm GW}_{(W,G)}(\tau, \bar{\tau})$ in the GW theory, with $t=2\pi i \tau/r$.
The above theorem then implies the commutative diagram in Fig. \ref{LGCYcorrespondence} (cf. \cite[Fig. 3]{Zhou:2014thesis}). The  commutativity of the diagram that we have shown using modular forms reflects
the invariance of the underlying Frobenius structures under the $\mathrm{SL}_{2}(\mathbb{C})$ action,
as previously studied by \cite{Dubrovin:1996}.

\begin{figure}[h]
  \renewcommand{\arraystretch}{1.2} 
\begin{displaymath}
\xymatrixcolsep{5pc}\xymatrixrowsep{5pc}\xymatrix{  \mathcal{A}^{\rm GW}_{(W,G)}(\tau) 
\ar@/^/[r]^{\textrm{modular completion}}\ar@{.>}[d]^{\mathscr{C}_{\mathrm{hol}}}   
 &  \ar@/^/[l]^{\textrm{constant term map}} \mathcal{A}^{\rm GW}_{(W,G)}(\tau,\bar{\tau})  \ar[d]^{\mathscr{C}} \\
 \mathcal{A}^{\rm FJRW}_{(W,G)}(u) \ar@/^/[r]^{\textrm{modular completion}}&  \ar@/^/[l]^{\textrm{holomorphic limit}}  \mathcal{A}^{\rm FJRW}_{(W,G)}(u,\bar{u})
 }
\end{displaymath}
  \caption[LGCYcorrespondence]{LG/CY correspondence}
  \label{LGCYcorrespondence}
\end{figure}

In the language of Givental's formalism \cite{Givental:20011}, the left vertical arrow can be stated as the following quantization formula
\begin{equation}
\mathscr{C}_{\rm hol}\left(\mathcal{A}^{\rm GW}_{(W,G)}(\tau)\right)=\widehat{\mathscr{G}^{-1}}
\left(\mathcal{A}^{\rm FJRW}_{(W,G)}(u)\right)\,.
\end{equation}
Here $\widehat{\mathscr{G}^{-1}}$ is the quantization operator from the Fock space with respect to $\mathcal{H}^{\rm GW}$ to the Fock space with respect to $\mathcal{H}^{\rm FJRW}$. We refer the interested readers to \cite{Givental:20011, Coates:2012} for more details about the construction of Fock spaces, quantization operator, and Givental's formalism.
\end{rem}

\subsection{LG/CY correspondence for $\mathbb{P}^1_{3,3,3}$}\label{sec-cubic}

Recall that in this case, we have $W=x_1^3+x_2^3+x_3^3$ and $G={\rm Aut}(W)$.
Let $\lambda=1/\sqrt{3}$, we construct a morphism 
\begin{equation*}
\mathscr{G}:\left(\mathcal{H}^{\rm GW},\eta^{\rm GW}\right)\to \left(\mathcal{H}^{\rm FJRW},\eta^{\rm FJRW}\right)
\end{equation*}
according to Tab. \ref{tablestatespaceiso333}.
\begin{table}[h]
  \caption[Isomorphism between state spaces for the $\mathbb{P}^{1}_{3,3,3}$ case]{Isomorphism between state spaces for the $\mathbb{P}^{1}_{3,3,3}$ case}
  \label{tablestatespaceiso333}
  \renewcommand{\arraystretch}{2.0} 
\begin{displaymath}
  \begin{tabular}{ |c | c  c c c cccc|}
    \hline
 $\bullet$ & $1$ & $\Delta_1$  & $\Delta_2$& $\Delta_3$ & $3\Delta_3^2$& $3\Delta_2^2$ & $3\Delta_1^2$ & $\mathcal{P}$ \\ 
\hline
   $\mathscr{G}(\bullet)$ & $1$ & $\lambda\phi_1$  & $\lambda\phi_2$ & $\lambda\phi_3$  & $\lambda^{-1}\phi_4$ & $\lambda^{-1}\phi_5$  & $\lambda^{-1}\phi_6$ & $\phi$ \\ 
\hline
  \end{tabular}
\end{displaymath}
\end{table}
It is easy to see that indeed the map $\mathscr{G}$ preserves the degree and pairing. \\

To simplify the notations, we denote
\begin{equation}\label{cubic-basic}
\LL\Delta_1,\Delta_2,\Delta_3\RR_{0,3}^{\rm GW}=M_1(\tau)\,, 
\LL\Delta_1,\Delta_1,\Delta_1\RR_{0,3}^{\rm GW}=M_2(\tau)\,, 
\LL\Delta_1,\Delta_1,\Delta_1^2,\Delta_1^2\RR_{0,4}^{\rm GW}=M_3(\tau)\,.
\end{equation}
We recall that in \cite{Shen:2014}, the WDVV equations satisfied by these functions are shown to be equivalent to the Ramanujan identities \eqref{ramanujan} for the modular group $\Gamma_{0}(3)$. 
By matching the boundary conditions, the following equalities are obtained 
\begin{equation}\label{gw-quasimodular}
C_{3}(\tau)=
3M_1(\tau), \quad
A_{3}(\tau)=
3M_2(\tau), \quad
E_{3}(\tau)=-9M_3(\tau).
\end{equation}
Now we have the following lemma (cf. \cite{Basalaev:20141, Basalaev:20142})
\begin{lem}\label{key-lemma}
The identity \eqref{lg-cy-correlator} holds for the basic correlation functions in \eqref{cubic-basic}. 
\end{lem}
\begin{proof}
By Proposition \ref{holomorphic-invariance}, the differential equations \eqref{ramanujan} for the quasi-modular forms $M_{1}(\tau)$, $M_{2}(\tau)$, $M_{3}(\tau)$ are invariant under the Cayley transformation $\mathscr{C}_{\rm hol}$. Comparing these equations with the WDVV equations in \eqref{eqWDVVFJRW333}, and the corresponding boundary conditions \eqref{333-initial} with \eqref{initial-333}, we obtain
\begin{eqnarray*}
&&\mathscr{C}_{\rm hol}\left(\LL\Delta_1,\Delta_2,\Delta_3\RR_{0,3}^{\rm GW}(\tau)\right)=\sqrt{3}\LL \lambda \phi_1, \lambda \phi_2,  \lambda\phi_3\RR^{\rm FJRW}_{0,3}(\tilde{\tau}_{\mathrm{orb}}),\\
&&\mathscr{C}_{\rm hol}\left(\LL\Delta_1,\Delta_1,\Delta_1\RR_{0,3}^{\rm GW}(\tau)\right)=\sqrt{3}\LL \lambda\phi_1, \lambda\phi_1, \lambda\phi_1\RR^{\rm FJRW}_{0,3}(\tilde{\tau}_{\mathrm{orb}}),\\
&&\mathscr{C}_{\rm hol}\left(\LL  \Delta_1, \Delta_1, 3\Delta_1^2, 3\Delta_1^2\RR_{0,4}^{\rm GW}(\tau)\right)=\sqrt{3}^2\LL \lambda \phi_1,  \lambda\phi_1,\lambda^{-1} \phi_6, \lambda^{-1}\phi_6\RR^{\rm FJRW}_{0,4}(\tilde{\tau}_{\mathrm{orb}})\,,
\end{eqnarray*}
where $\tilde{\tau}_{\mathrm{orb}}=M_{\mathrm{rational}}\mathcal{C}(\tau)$ is given in \eqref{eqrationalparameter}. We now apply a further change of normalization in \eqref{eqchangofnormalization} with $M=3$, that is, we consider the scaled Cayley transform $\mathcal{C}: \tau \mapsto u=3M_{\mathrm{rational}}\mathcal{C}(\tau)$.
Then according to \eqref{eqCayleyfromnormalization},
the resulting new Cayley transformation gives the following desired identities
\begin{eqnarray*}
&&\mathscr{C}_{\rm hol}\left(\LL\Delta_1,\Delta_2,\Delta_3\RR_{0,3}^{\rm GW}(\tau)\right)=\LL \lambda \phi_1, \lambda \phi_2,  \lambda\phi_3\RR^{\rm FJRW}_{0,3}(u),\\
&&\mathscr{C}_{\rm hol}\left(\LL\Delta_1,\Delta_1,\Delta_1\RR_{0,3}^{\rm GW}(\tau)\right)=\LL \lambda\phi_1, \lambda\phi_1, \lambda\phi_1\RR^{\rm FJRW}_{0,3}(u),\\
&&\mathscr{C}_{\rm hol}\left(\LL  \Delta_1, \Delta_1, 3\Delta_1^2, 3\Delta_1^2\RR_{0,4}^{\rm GW}(\tau)\right)=\LL \lambda \phi_1,  \lambda\phi_1,\lambda^{-1} \phi_6, \lambda^{-1}\phi_6\RR^{\rm FJRW}_{0,4}(u)\,.
\end{eqnarray*}
\end{proof}
\begin{proof}[Proof of Theorem \ref{lg-cy-main} for the $\mathbb{P}^{1}_{3,3,3}$ case]
As described in \cite{Krawitz:2011}, both GW correlation functions and FJRW correlation functions in these cases can be uniquely reconstructed from the pairing, the ring structure constants, and the correlation functions described in Lemma \ref{key-lemma} (hence the name building blocks). 
Lemma \ref{key-lemma} tells that the Cayley transformation $\mathscr{C}_{\rm hol}$ matches the building blocks in the two theories.
Since the reconstruction process of these two theories are the same, we only need to check $\mathscr{C}_{\rm hol}$ is compatible with the process of reconstruction which only involves addition, multiplication and differentiation. Then Theorem \ref{lg-cy-main} follows from Proposition \ref{thminvariancequasimodular} that the Cayley transformation preserves the differential ring structure.

We finish the proof by further checking the reconstruction process in three steps.
\paragraph{Step One: Primary correlation functions in genus $g=0$.}
This follows from Lemma \ref{key-lemma} and comparing the FJRW prepotential in Proposition \ref{fjrw-e3} with the formula for $\mathcal{F}^{\rm GW}_{0,\mathbb{P}^1_{3,3,3}}$ in \cite{Shen:2014}. 
\paragraph{Step Two:
Primary correlation functions in genus $g\geq1$.}
On the FJRW side, we look at the degree formula \eqref{degree-axiom}. Note $\hat{c}_W=1$ and $\deg_W(\phi_j)\leq 1.$ If $g\geq1$ and each $\ell_j=0$, then \eqref{degree-axiom} holds only if $g=1$ and all $\deg_W(\phi_j)=1$. Thus on the FJRW side, any non-vanishing correlation function must be a
differential polynomial in
$$\LL\phi\RR_{1,1}^{\mathrm{FJRW}}\,.$$
Similarly, on the GW side, we only need to consider differential polynomials in
$$\LL \mathcal{P}\RR_{1,1}^{\mathrm{GW}}\,.$$
Because $\mathscr{C}_{\rm hol}$ is compatible with differentiation, we only need to check \eqref{lg-cy-correlator} for these two correaltion functions.

On the GW side, we integrate the GW class $\Lambda_{1,4}^{\mathrm{GW}}(\Delta_1^2,\Delta_1^2,\Delta_2,\Delta_3)$ over Getzler's relation \cite{Getzler:1997}
\begin{equation}\label{eq:Getzler}
  12\delta_{2,2}-4\delta_{2,3}-2\delta_{2,4}+6\delta_{3,4}+\delta_{0,3}+\delta_{0,4}-2\delta_{\beta}=0\in H_4(\overline{\mathcal{M}}_{1,4},\mathbb{Q})\,.
\end{equation}
It follows from the Splitting Axiom and genus zero computation that $\langle\langle \mathcal{P}\rangle \rangle_{1,1}^{\mathrm{GW}}$ is a linear combination of some genus zero primary correlation functions. The calculation in \cite{Shen:2014} showed\footnote{This corrects an error made in \cite{Shen:2014}
which differed from the correct result here by a factor of $2$.}
\begin{equation*}
\LL \mathcal{P}\RR_{1,1}^{\mathrm{GW}}=\frac{-2E_3+A_3^2}{24}\,.
\end{equation*}
Similarly, we can calculate $\langle\langle \phi\rangle\rangle_{1,1}^{\mathrm{FJRW}}$ by integrating $\Lambda_{1,4}^{\mathrm{FJRW}}(\phi_1,\phi_2,\phi_4,\phi_4)$ over the relation in \eqref{eq:Getzler}.
The Splitting Axiom and genus zero computation in Proposition \ref{fjrw-e3} allow us to obtain
\begin{equation*}
\LL\phi\RR_{1,1}^{\mathrm{FJRW}}=\frac{6f_{3}+f_{2}^2}{72}\,.
\end{equation*}
Then we can apply Lemma  \ref{key-lemma} to get
\begin{equation*}
\mathscr{C}_{\rm hol}\left(\LL  \mathcal{P}\RR_{1,1}^{\mathrm{GW}}\right)
=\mathscr{C}_{\rm hol}\left(\frac{-2E_3+A_3^2}{24}\right)=\frac{6f_{3}+f_{2}^2}{72}=\LL\phi\RR_{1,1}^{\mathrm{FJRW}}\,.
\end{equation*}

\paragraph{Step Three: Correlation functions with $\psi$-classes.}
To reconstruct these correlation functions, we recall the \emph{$g$-reduction technique} introduced in \cite{Faber:2010}. According to Ionel \cite{Ionel:2002} and Faber--Pandharipande \cite{Faber:2005},
if $M(\psi, \kappa)$ is a monomial of $\psi$-classes and $\kappa$-classes with $\deg M\geq g$ for $g\geq1$ or $\deg M\geq1$ for $g=0$, then $M(\psi, \kappa)$ can be presented by a linear combination of dual graphs on the boundary of $\overline{\mathcal{M}}_{g,k}$.
Now let us consider a correlation function (either GW or FJRW)
\begin{equation*}
\LL\phi_1\psi_1^{\ell_1},\cdots,\phi_k\psi_k^{\ell_k}\RR_{g,k}\,.
\end{equation*}
If 
\begin{equation}\label{g-nonvanishing}
\deg\left(\prod_{j=1}^{k}\psi_j^{\ell_j}\right)
:=\sum_{j=1}^{k}\ell_j\geq
\left\{
\begin{array}{ll}
2g-2,& g\geq1\,,\\
1,& g=0\,,
\end{array}
\right.
\end{equation}
then $\prod_{j=1}^{k}\psi_j^{\ell_j}$ is a monomial satisfying the condition for the $g$-reduction, thus we can apply this technique to reduce its degree.
If \eqref{g-nonvanishing} is not satisfied, then the correlation function will vanish by the degree formulas.
Hence by repeatedly using the Splitting Axiom, the vanishing condition, and the $g$-reduction, one can eventually rewrite 
the correlation function as a product of genus zero and genus one primary correlation functions. Now the result follows from the first two steps and the fact that $\mathscr{G}$ is pairing-preserving and is compatible with the splitting.
\end{proof}

\subsection{LG/CY correspondence for $\mathbb{P}^1_{2,2,2,2}$}

In this case, we have $W=x_1^4+x_2^4+x_3^2$ and $G=G_1\times {\rm Aut}(x_3^2)$, where $G_1=\langle (\sqrt{-1},\sqrt{-1}), (1,-1)\rangle$.
We first  apply a change of basis within the FJRW state space. Let $\phi_{0}=\tilde{\phi}_{0},\phi=\tilde{\phi}$, and
\begin{equation*}
{1\over \sqrt{2}}
\begin{pmatrix}
\phi_1,
\phi_3
\end{pmatrix}
\begin{pmatrix}
e^{3\pi i\over 4}&e^{5\pi i\over 4}\\
e^{5\pi i\over 4}&e^{3\pi i\over 4}
\end{pmatrix}
=\pm
\begin{pmatrix}
\tilde{\phi}_{1},
\tilde{\phi}_{3}
\end{pmatrix},
\quad
{1\over \sqrt{2}}
\begin{pmatrix}
\phi_2,
\phi_4
\end{pmatrix}
\begin{pmatrix}
1&1\\
1&-1
\end{pmatrix}
=\pm\begin{pmatrix}
\tilde{\phi}_{2},
\tilde{\phi}_{4}
\end{pmatrix}\,.
\end{equation*}
It is easy to check that the non-vanishing parings in the new basis are given by
\begin{equation*}
\eta^{\rm FJRW}(\tilde{\phi}_{0}, \tilde{\phi})=\eta^{\rm FJRW}( \tilde{\phi}_i, \tilde{\phi}_i)=1\,,\quad i=1,2,3,4\,.
\end{equation*} 
Denote the corresponding coordinates with respect to the basis $\{\tilde{\phi}_{i}\}_{i=0}^{5}$ by $\{v_{i}\}_{i=0}^{5}$. In the new coordinate system, the prepotential $\F_{0,(W,G)}^{\rm FJRW}(u)$ in Proposition \ref{pillow-fjrw} becomes
\begin{equation}\label{new-pillow}
\begin{aligned}
\F_{0,(W,G)}^{\rm FJRW}(v)=&
~\mathrm{cubic~ terms}\\
&+(-f_1(v)) v_1 v_2 v_3 v_4+
 (2f_2(v)+f_{3}(v))\frac{1}{4!}\sum_{i}v_i^4+f_3(v)\frac{1}{2!2!}\sum_{i<j}v_i^2 v_j^2\,.
\end{aligned}
\end{equation}
\begin{proof}[Proof of Theorem \ref{lg-cy-main} for the $\mathbb{P}^{1}_{2,2,2,2}$ case]
We construct the linear map $\mathscr{G}:\left(\mathcal{H}^{\rm GW},\eta^{\rm GW}\right)\to \left(\mathcal{H}^{\rm FJRW},\eta^{\rm FJRW}\right)$
according to Tab. \ref{tablestatespaceiso2222}.
\begin{table}[h]
  \caption[Isomorphism between state spaces for the $\mathbb{P}^{1}_{2,2,2,2}$ case]{Isomorphism between state spaces for the $\mathbb{P}^{1}_{2,2,2,2}$ case}
  \label{tablestatespaceiso2222}
  \renewcommand{\arraystretch}{2.0} 
\begin{displaymath}
  \begin{tabular}{ |c | c  c c ccc|}
    \hline
 $\bullet$ & $1$ & $\sqrt{2}\Delta_1$  & $ \sqrt{2}\Delta_2$& $\sqrt{2}\Delta_3$ & $\sqrt{2}\Delta_4$& $\mathcal{P}$ \\ 
\hline
   $\mathscr{G}(\bullet)$ & $1$ & $\tilde{\phi}_1$  & $\tilde{\phi}_2$ & $\tilde{\phi}_3$  & $\tilde{\phi}_4$ & $\tilde{\phi}$ \\ 
\hline
  \end{tabular}
\end{displaymath}
\end{table}
It is obvious that $\mathscr{G}$ preserves the pairing. 

Recall that on the GW side,  the genus zero primary potential is \cite{Satake:2011, Shen:2014}
\begin{equation*}\label{eqg0N2}
\mathcal{F}_{0,(W,G)}^{\mathrm{GW}}=\mathrm{cubic~ terms}+ X(\tau)t_1t_2t_3t_4 + \frac{Y(\tau)}{4!} (\sum_{i=1}^4t_i^4)+ \frac{Z(\tau)}{2!2!} \sum_{i<j}(t_i^2 t_j^2)\,,
\end{equation*}
where the correlation functions $X(\tau),Y(\tau),Z(\tau)$ are the following quasi-modular forms for $\Gamma_{0}(2)$ (see Section \ref{secperiodcalculation} for details of the generators $A_{2},C_{2},E_{2}$)
\begin{equation*}
X(\tau)={1\over 8}C_{2}^{2}(\tau)\,,\quad Y(\tau)=-{1\over 16}(3E_{2}(\tau)+A_{2}^{2}(\tau))\,,
\quad Z(\tau)={1\over 16}(-E_{2}(\tau)+A_{2}^{2}(\tau))\,.
\end{equation*}
The equations satisfied by $\mathscr{C}_{\mathrm{hol}}(X(\tau)), \mathscr{C}_{\mathrm{hol}}(Y(\tau)),\mathscr{C}_{\mathrm{hol}}(Z(\tau))$ can be derived from \eqref{ramanujan} and Theorem \ref{thminvariancequasimodular}, the boundary conditions
from \eqref{eq2222boundary}.
Now we compare them with the WDVV equations  \eqref{pillow-wdvv-1} for the FJRW correlations functions, and the corresponding boundary conditions in \eqref{pillow-initial}, we are led to
\begin{equation*}\label{pillow-genus-zero}
\mathscr{C}_{\rm hol}\big(2X(\tau)\big)=-f_1(\tilde{\tau}_{\mathrm{orb}}), \quad \mathscr{C}_{\rm hol}\big(2Y(\tau)\big)=2f_2(\tilde{\tau}_{\mathrm{orb}})+f_3(\tilde{\tau}_{\mathrm{orb}}), \quad \mathscr{C}_{\rm hol}\big(2Z(\tau)\big)=f_3(\tilde{\tau}_{\mathrm{orb}})\,.
\end{equation*}
Hence we obtain, under the Cayley transform $\tau\mapsto \tilde{\tau}_{\mathrm{orb}}=M_{\mathrm{rational}}\mathcal{C}(\tau)$,
\begin{eqnarray*}
&&\mathscr{C}_{\mathrm{hol}}\left(\LL  \sqrt{2}\Delta_1, \sqrt{2}\Delta_2, \sqrt{2}\Delta_3,  \sqrt{2}\Delta_{4}\RR_{0,4}^{\rm GW}\right)=2\LL \tilde{\phi}_1, \tilde{\phi}_2,\tilde{\phi}_3,\tilde{\phi}_4\RR^{\rm FJRW}_{0,4}(\tilde{\tau}_{\mathrm{orb}})\,,\\
&&\mathscr{C}_{\mathrm{hol}}\left(\LL  \sqrt{2}\Delta_1, \sqrt{2}\Delta_1, \sqrt{2}\Delta_1,  \sqrt{2}\Delta_{1}\RR_{0,4}^{\rm GW}\right)=2\LL \tilde{\phi}_1, \tilde{\phi}_1,\tilde{\phi}_1,\tilde{\phi}_1\RR^{\rm FJRW}_{0,4}(\tilde{\tau}_{\mathrm{orb}})\,,\\
&&\mathscr{C}_{\mathrm{hol}}\left(\LL  \sqrt{2}\Delta_1, \sqrt{2}\Delta_1, \sqrt{2}\Delta_2,  \sqrt{2}\Delta_{2}\RR_{0,4}^{\rm GW}\right)=2\LL \tilde{\phi}_1, \tilde{\phi}_1,\tilde{\phi}_2,\tilde{\phi}_2\RR^{\rm FJRW}_{0,4}(\tilde{\tau}_{\mathrm{orb}})\,.
\end{eqnarray*}
Note that all of the quasi-modular forms involved in the genus zero primary potential have weight $2$.
Similar to the $\mathbb{P}^{1}_{3,3,3}$ case, by scaling the coordinate
\begin{equation}
\tilde{\tau}_{\mathrm{orb}}=M_{\mathrm{rational}}\mathcal{C}(\tau)\mapsto 2\tilde{\tau}_{\mathrm{orb}}\,,
\end{equation}
the resulting new Cayley transform $\tau\mapsto 2\tilde{\tau}_{\mathrm{orb}}$ matches the correlation functions exactly.
That is, by equating $v_{i}$ with $t_{i},i=0,\cdots 4$, and $v$ with $2\tilde{\tau}_{\mathrm{orb}}$, we have
\begin{equation}
\mathscr{C}_{\mathrm{hol}}(\mathcal{F}_{0,(W,G)}^{\mathrm{GW}}(\tau))=\F_{0,(W,G)}^{\rm FJRW}(v)\,.
\end{equation}
The rest of the proof follows from the same reasoning in the $\mathbb{P}^{1}_{3,3,3}$ case.
\end{proof}

\begin{appendices}
\numberwithin{equation}{section}

\section{WDVV equations in FJRW theories}
\label{appendixA}

\paragraph*{The WDVV equations in \eqref{wdvv-333} are explained graphically as follows:}
\begin{itemize}
\item $f_1f_4=f_2f_5; \quad  (i,j,k,\ell)=(1,2,1,2),  \quad S=\{4\}.$
\begin{center}
\begin{picture}(120,23)


\put(27, 8){$=$}
\put(-23, 8){$2\bigg($}
\put(22, 8){$\bigg)$}
\put(83, 8){$+$}

	\put(-10,9){\line(-3,4){5}}
    \put(-10,9){\line(-3,-4){5}}
	\put(-10,9){\line(1,0){10}}
	\put(0,9){\line(1,0){10}}
    \put(10,9){\line(3,4){5}}
    \put(10,9){\line(3,-4){5}}

	
	\put(-17,17){$\phi_1$}
	\put(-17,0){$\phi_2$}
        \put(-6,10){$\phi_3$}
        \put(5,10){$\phi_4$}
	\put(15,17){$\phi_1$}
        \put(16,8){$\phi_4$}
	\put(15,0){$\phi_2$}

    \put(10,9){\line(3,0){5}}

	\put(50,9){\line(-3,4){5}}
    \put(50,9){\line(-3,-4){5}}
	\put(50,9){\line(1,0){10}}
	\put(60,9){\line(1,0){10}}
    \put(70,9){\line(3,4){5}}
    \put(70,9){\line(3,-4){5}}

	
	\put(43,17){$\phi_1$}
        \put(40,8){$\phi_4$}
	\put(43,0){$\phi_1$}
        \put(54,10){$\phi_5$}
        \put(65,10){$\phi_2$}
	\put(75,17){$\phi_2$}
	\put(75,0){$\phi_2$}


    \put(50,9){\line(-2,0){5}}

	\put(100,9){\line(-3,4){5}}
    \put(100,9){\line(-3,-4){5}}
	\put(100,9){\line(1,0){10}}
	\put(110,9){\line(1,0){10}}
    \put(120,9){\line(3,4){5}}
    \put(120,9){\line(3,-4){5}}

	
	\put(93,17){$\phi_1$}
	\put(93,0){$\phi_1$}
        \put(104,10){$\phi_1$}
        \put(115,10){$\phi_6$}
	\put(125,17){$\phi_2$}
        \put(126,8){$\phi_4$}
	\put(125,0){$\phi_2$}


    \put(120,9){\line(2,0){5}}

\end{picture}
\end{center}
\item $f_1f_5=f_2'+f_2f_6; \quad  (i,j,k,\ell)=(1,2,1,5),  \quad S=\{1\}.$

\begin{center}

\begin{picture}(120,23)


\put(27, 8){$=$}
\put(83, 8){$+$}

	\put(50,9){\line(-3,4){5}}
    \put(50,9){\line(-3,-4){5}}
	\put(50,9){\line(1,0){10}}
	\put(60,9){\line(1,0){10}}
    \put(70,9){\line(3,4){5}}
    \put(70,9){\line(3,-4){5}}

	
	\put(43,17){$\phi_1$}
        \put(40,8){$\phi_1$}
	\put(43,0){$\phi_1$}
        \put(54,10){$\phi_7$}
        \put(65,10){$\phi_0$}
	\put(75,17){$\phi_2$}
	\put(75,0){$\phi_5$}

    \put(50,9){\line(-2,0){5}}

	\put(100,9){\line(-3,4){5}}
    \put(100,9){\line(-3,-4){5}}
	\put(100,9){\line(1,0){10}}
	\put(110,9){\line(1,0){10}}
    \put(120,9){\line(3,4){5}}
    \put(120,9){\line(3,-4){5}}

	
	\put(93,17){$\phi_1$}
	\put(93,0){$\phi_1$}
        \put(104,10){$\phi_1$}
        \put(115,10){$\phi_6$}
	\put(125,17){$\phi_2$}
        \put(126,8){$\phi_1$}
	\put(125,0){$\phi_5$}

    \put(120,9){\line(2,0){5}}

	\put(-10,9){\line(-3,4){5}}
    \put(-10,9){\line(-3,-4){5}}
	\put(-10,9){\line(1,0){10}}
	\put(0,9){\line(1,0){10}}
    \put(10,9){\line(3,4){5}}
    \put(10,9){\line(3,-4){5}}

	
	\put(-17,17){$\phi_1$}
	\put(-17,0){$\phi_2$}
        \put(-6,10){$\phi_3$}
        \put(5,10){$\phi_4$}
	\put(15,17){$\phi_1$}
        \put(16,8){$\phi_1$}
	\put(15,0){$\phi_5$}

    \put(10,9){\line(3,0){5}}

\end{picture}
\end{center}

\item $2f_1f_6=f_3f_1+f_2f_4; \quad  (i,j,k,\ell)=(1,2,1,3),  \quad S=\{6\}$.

\begin{center}
\setlength{\unitlength}{0.08cm}
\begin{picture}(60,23)


\put(27, 8){$=$}
\put(83, 8){$+$}
\put(-27, 8){$+$}

	\put(-60,9){\line(-3,4){5}}
    \put(-60,9){\line(-3,-4){5}}
	\put(-60,9){\line(1,0){10}}
	\put(-50,9){\line(1,0){10}}
    \put(-40,9){\line(3,4){5}}
    \put(-40,9){\line(3,-4){5}}

	
	\put(-67,17){$\phi_1$}
        \put(-70,8){$\phi_6$}
	\put(-67,0){$\phi_2$}
        \put(-56,10){$\phi_5$}
        \put(-45,10){$\phi_2$}
	\put(-35,17){$\phi_1$}
	\put(-35,0){$\phi_3$}

    \put(-60,9){\line(-2,0){5}}

	\put(50,9){\line(-3,4){5}}
    \put(50,9){\line(-3,-4){5}}
	\put(50,9){\line(1,0){10}}
	\put(60,9){\line(1,0){10}}
    \put(70,9){\line(3,4){5}}
    \put(70,9){\line(3,-4){5}}

	
	\put(43,17){$\phi_1$}
        \put(40,8){$\phi_6$}
	\put(43,0){$\phi_1$}
        \put(54,10){$\phi_6$}
        \put(65,10){$\phi_1$}
	\put(75,17){$\phi_2$}
	\put(75,0){$\phi_3$}

    \put(50,9){\line(-2,0){5}}

	\put(100,9){\line(-3,4){5}}
    \put(100,9){\line(-3,-4){5}}
	\put(100,9){\line(1,0){10}}
	\put(110,9){\line(1,0){10}}
    \put(120,9){\line(3,4){5}}
    \put(120,9){\line(3,-4){5}}

	
	\put(93,17){$\phi_1$}
	\put(93,0){$\phi_1$}
        \put(104,10){$\phi_1$}
        \put(115,10){$\phi_6$}
	\put(125,17){$\phi_2$}
        \put(126,8){$\phi_6$}
	\put(125,0){$\phi_3$}

    \put(120,9){\line(2,0){5}}

	\put(-10,9){\line(-3,4){5}}
    \put(-10,9){\line(-3,-4){5}}
	\put(-10,9){\line(1,0){10}}
	\put(0,9){\line(1,0){10}}
    \put(10,9){\line(3,4){5}}
    \put(10,9){\line(3,-4){5}}

	
	\put(-17,17){$\phi_1$}
	\put(-17,0){$\phi_2$}
        \put(-6,10){$\phi_3$}
        \put(5,10){$\phi_4$}
	\put(15,17){$\phi_1$}
        \put(16,8){$\phi_6$}
	\put(15,0){$\phi_3$}

    \put(10,9){\line(3,0){5}}

\end{picture}
\end{center}

\item $f_1f_5'=2f_5f_1'; \quad  (i,j,k,\ell)=(1,2,5,7),  \quad S=\{1,1\}.$
\begin{center}
\begin{picture}(50,23)


\put(27, 8){$=$}
\put(34, 8){$2\bigg($}
\put(82, 8){$\bigg)$}

	\put(50,9){\line(-3,4){5}}
    \put(50,9){\line(-3,-4){5}}
	\put(50,9){\line(1,0){10}}
	\put(60,9){\line(1,0){10}}
    \put(70,9){\line(3,4){5}}
    \put(70,9){\line(3,-4){5}}

	
	\put(43,17){$\phi_1$}
	\put(40,8){$\phi_1$}
	\put(43,0){$\phi_5$}
        \put(54,10){$\phi_4$}
        \put(65,10){$\phi_3$}
	\put(75,17){$\phi_2$}
	\put(76,8){$\phi_1$}
	\put(75,0){$\phi_7$}

    \put(50,9){\line(-2,0){5}}

    \put(70,9){\line(2,0){5}}

	\put(-10,9){\line(-3,4){5}}
    \put(-10,9){\line(-3,-4){5}}
	\put(-10,9){\line(1,0){10}}
	\put(0,9){\line(1,0){10}}
    \put(10,9){\line(3,4){5}}
    \put(10,9){\line(3,-4){5}}

	
	\put(-17,17){$\phi_1$}
	\put(-17,0){$\phi_2$}
        \put(-6,10){$\phi_3$}
        \put(5,10){$\phi_4$}
	\put(15,17){$\phi_5$}
        \put(16,6){$\phi_1$}
        \put(16,11){$\phi_1$}
	\put(15,0){$\phi_7$}

    \put(10,9){\line(3,2){5}}
    \put(10,9){\line(3,-2){5}}

\end{picture}
\end{center}

\item $f_1f_3'=2f_6f_1'; \quad  (i,j,k,\ell)=(2,3,4,7),  \quad S=\{1,1\}.$
\begin{center}
\begin{picture}(50,23)


\put(27, 8){$=$}
\put(34, 8){$2\bigg($}
\put(82, 8){$\bigg)$}

	\put(50,9){\line(-3,4){5}}
    \put(50,9){\line(-3,-4){5}}
	\put(50,9){\line(1,0){10}}
	\put(60,9){\line(1,0){10}}
    \put(70,9){\line(3,4){5}}
    \put(70,9){\line(3,-4){5}}

	
	\put(43,17){$\phi_2$}
	\put(40,8){$\phi_1$}
	\put(43,0){$\phi_4$}
        \put(54,10){$\phi_5$}
        \put(65,10){$\phi_2$}
	\put(75,17){$\phi_3$}
	\put(76,8){$\phi_1$}
	\put(75,0){$\phi_7$}

    \put(50,9){\line(-2,0){5}}

    \put(70,9){\line(2,0){5}}

	\put(-10,9){\line(-3,4){5}}
    \put(-10,9){\line(-3,-4){5}}
	\put(-10,9){\line(1,0){10}}
	\put(0,9){\line(1,0){10}}
    \put(10,9){\line(3,4){5}}
    \put(10,9){\line(3,-4){5}}

	
	\put(-17,17){$\phi_2$}
	\put(-17,0){$\phi_3$}
        \put(-6,10){$\phi_1$}
        \put(5,10){$\phi_6$}
	\put(15,17){$\phi_4$}
        \put(16,6){$\phi_1$}
        \put(16,11){$\phi_1$}
	\put(15,0){$\phi_7$}

    \put(10,9){\line(3,2){5}}
    \put(10,9){\line(3,-2){5}}

\end{picture}
\end{center}

\end{itemize}

\paragraph*{The WDVV equations in \eqref{pillow-wdvv-1} are explained graphically as follows:}
\begin{itemize}

\item $f_1'+2f_1f_2=0, \quad  (i,j,k,\ell)=(2,2,1,1),  \quad S=\{4,4\}$.

\begin{center}
\begin{picture}(120,23)

	\put(-10,9){\line(-3,4){5}}
    \put(-10,9){\line(-3,-4){5}}
	\put(-10,9){\line(1,0){10}}
	\put(0,9){\line(1,0){10}}
    \put(10,9){\line(3,4){5}}
    \put(10,9){\line(3,-4){5}}

	
	\put(-17,17){$\phi_2$}
	\put(-17,0){$\phi_2$}
        \put(-6,10){$\phi_0$}
        \put(5,10){$\phi_5$}
	\put(15,17){$\phi_1$}
        \put(16,6){$R$}
        \put(16,11){$R$}
	\put(15,0){$\phi_1$}

    \put(10,9){\line(3,2){5}}
    \put(10,9){\line(3,-2){5}}


\put(27, 8){$+$}
\put(83, 8){$=$}

	\put(50,9){\line(-3,4){5}}
    \put(50,9){\line(-3,-4){5}}
	\put(50,9){\line(1,0){10}}
	\put(60,9){\line(1,0){10}}
    \put(70,9){\line(3,4){5}}
    \put(70,9){\line(3,-4){5}}

	
	\put(43,17){$\phi_2$}
        \put(40,8){$R$}
	\put(43,0){$\phi_2$}
        \put(54,10){$R$}
        \put(65,10){$R$}
	\put(75,17){$\phi_1$}
	\put(76,8){$R$}
	\put(75,0){$\phi_1$}

    \put(50,9){\line(-2,0){5}}
    \put(70,9){\line(2,0){5}}

	\put(100,9){\line(-3,4){5}}
    \put(100,9){\line(-3,-4){5}}
	\put(100,9){\line(1,0){10}}
	\put(110,9){\line(1,0){10}}
    \put(120,9){\line(3,4){5}}
    \put(120,9){\line(3,-4){5}}

	
	\put(93,17){$\phi_1$}
	\put(93,0){$\phi_1$}
        \put(104,10){$\phi_1$}
        \put(115,10){$\phi_6$}
	\put(125,17){$\phi_2$}
        \put(126,8){$\phi_1$}
	\put(125,0){$\phi_5$}

    \put(120,9){\line(2,0){5}}

\end{picture}
\end{center}

\item $f_2'+2f_2f_3+f_3'=0, \quad (i,j,k,\ell)=(1,3,2,2),  \quad S=\{4,4\}$.

\begin{center}
\begin{picture}(120,23)

	\put(-10,9){\line(-3,4){5}}
    \put(-10,9){\line(-3,-4){5}}
	\put(-10,9){\line(1,0){10}}
	\put(0,9){\line(1,0){10}}
    \put(10,9){\line(3,4){5}}
    \put(10,9){\line(3,-4){5}}

	
	\put(-17,17){$\phi_1$}
	\put(-17,0){$\phi_3$}
        \put(-6,10){$\phi_0$}
        \put(5,10){$\phi_5$}
	\put(15,17){$\phi_2$}
        \put(16,6){$R$}
        \put(16,11){$R$}
	\put(15,0){$\phi_2$}

    \put(10,9){\line(3,2){5}}
    \put(10,9){\line(3,-2){5}}


\put(27, 8){$+2\bigg($}
\put(83, 8){$\bigg)+$}
\put(127, 8){$=0$}

	\put(50,9){\line(-3,4){5}}
    \put(50,9){\line(-3,-4){5}}
	\put(50,9){\line(1,0){10}}
	\put(60,9){\line(1,0){10}}
    \put(70,9){\line(3,4){5}}
    \put(70,9){\line(3,-4){5}}

	
	\put(43,17){$\phi_1$}
        \put(40,8){$R$}
	\put(43,0){$\phi_3$}
        \put(54,10){$R$}
        \put(65,10){$R$}
	\put(75,17){$\phi_2$}
	\put(76,8){$R$}
	\put(75,0){$\phi_2$}

    \put(50,9){\line(-2,0){5}}
    \put(70,9){\line(2,0){5}}

	\put(100,9){\line(-3,4){5}}
    \put(100,9){\line(-3,-4){5}}
	\put(100,9){\line(1,0){10}}
	\put(110,9){\line(1,0){10}}
    \put(120,9){\line(3,4){5}}
    \put(120,9){\line(3,-4){5}}

	
	\put(93,17){$\phi_1$}
	\put(91,6){$R$}
	\put(91,11){$R$}
	\put(93,0){$\phi_3$}
        \put(104,10){$\phi_7$}
        \put(115,10){$\phi_0$}
	\put(125,17){$\phi_2$}
	\put(125,0){$\phi_2$}

    \put(100,9){\line(-3,2){5}}
    \put(100,9){\line(-3,-2){5}}


\end{picture}
\end{center}

\item $f_3'+f_3^2=f_1^2, \quad (i,j,k,\ell)=(1,3,1,3), \quad  S=\{4,4\}.$

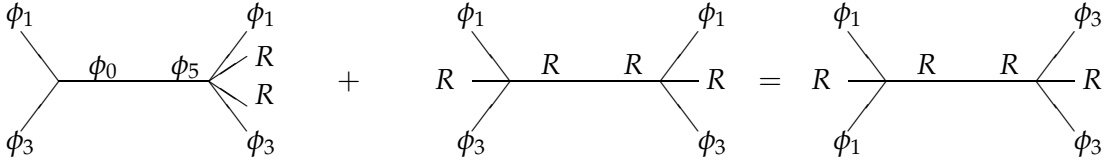
\begin{center}
\begin{picture}(120,23)

	\put(-10,9){\line(-3,4){5}}
    \put(-10,9){\line(-3,-4){5}}
	\put(-10,9){\line(1,0){10}}
	\put(0,9){\line(1,0){10}}
    \put(10,9){\line(3,4){5}}
    \put(10,9){\line(3,-4){5}}

	
	\put(-17,17){$\phi_1$}
	\put(-17,0){$\phi_3$}
        \put(-6,10){$\phi_0$}
        \put(5,10){$\phi_5$}
	\put(15,17){$\phi_1$}
        \put(16,6){$R$}
        \put(16,11){$R$}
	\put(15,0){$\phi_3$}

    \put(10,9){\line(3,2){5}}
    \put(10,9){\line(3,-2){5}}


\put(27, 8){$+$}
\put(83, 8){$=$}

	\put(50,9){\line(-3,4){5}}
    \put(50,9){\line(-3,-4){5}}
	\put(50,9){\line(1,0){10}}
	\put(60,9){\line(1,0){10}}
    \put(70,9){\line(3,4){5}}
    \put(70,9){\line(3,-4){5}}

	
	\put(43,17){$\phi_1$}
        \put(40,8){$R$}
	\put(43,0){$\phi_3$}
        \put(54,10){$R$}
        \put(65,10){$R$}
	\put(75,17){$\phi_1$}
	\put(76,8){$R$}
	\put(75,0){$\phi_3$}

    \put(50,9){\line(-2,0){5}}
    \put(70,9){\line(2,0){5}}

	\put(100,9){\line(-3,4){5}}
    \put(100,9){\line(-3,-4){5}}
	\put(100,9){\line(1,0){10}}
	\put(110,9){\line(1,0){10}}
    \put(120,9){\line(3,4){5}}
    \put(120,9){\line(3,-4){5}}

	
	\put(93,17){$\phi_1$}
        \put(90,8){$R$}
	\put(93,0){$\phi_1$}
        \put(104,10){$R$}
        \put(115,10){$R$}
	\put(125,17){$\phi_3$}
        \put(126,8){$R$}
	\put(125,0){$\phi_3$}

    \put(100,9){\line(-2,0){5}}
    \put(120,9){\line(2,0){5}}

\end{picture}
\end{center}

\end{itemize}

\section{Analytic continuation of periods and connection matrices}
\label{appendixB}

In this appendix we discuss the analytic continuation of the periods for the elliptic curve families mentioned in Section \ref{secperiodcalculation}.
Most of the material presented here is classical  (see e.g., \cite{Fre:1995bc} for the monodromy calculations for the Hesse pencil) and are collected here for reference.

In the following we give the details for the Hesse pencil of elliptic curves. 
The discussions for the other elliptic curve families are similar.
In this case, the Picard-Fuchs operator is given by the Gauss hypergeometric one in \eqref{eqPFaroundinfinity}, with $a=1/3,b=2/3,c=1$,
\begin{equation}
\mathcal{L}=\theta_{\alpha}^{2}-\alpha(\theta_{\alpha}+{1\over 3})(\theta_{\alpha}+{2\over 3})\,, \quad \theta_{\alpha}:=\alpha{\partial \over \partial \alpha}\,.
\end{equation}

\paragraph{Local basis of solutions near $\alpha=0$}

The basis for the space of periods that we take near the singularity $\alpha=0$ on the base of the elliptic curve family or equivalently the cusp $[\tau]=[i\infty]$ on the modular curve, with the Hauptmodul given in \eqref{eqnHauptmodul}, is
\begin{eqnarray}
\pi_{1}&=&~_{2}F_{1}({1\over 3},{2\over 3},1,\alpha)\,,\nonumber\\
\kappa \pi_{2}&=&\kappa~_{2}F_{1}({1\over 3},{2\over 3},1,1-\alpha)\,.
\end{eqnarray}
Here we have followed the convention in \cite{Erdelyi:1981}  for the ordering of the solutions to hypergeometric differential equations. Our notations $\pi_{i},i=1,2\cdots$ correspond to the ones $u_{i},i=1,2\cdots$ therein.
The coefficient 
\begin{equation}
\kappa={1\over -2\pi i}{\Gamma(a)\Gamma(b)\over \Gamma(a+b)}= {i \over \sqrt{3}}
\end{equation} 
is fixed such that 
the local monodromy $M_{\infty}^{\mathrm{local}}: (\pi_{1},\kappa \pi_{2})\mapsto (\pi_{1},\kappa \pi_{2}) M_{\infty}^{\mathrm{local}}$ at $[\tau]=[i\infty]$ is
\begin{equation}
M_{\infty}^{\mathrm{local}}=
 \left( \begin{array}{cc}
1 & 1 \\
0& 1
\end{array} \right)=T\in \mathrm{SL}_{2}(\mathbb{Z})\,.
\end{equation}
See \cite[Page 110]{Erdelyi:1981} for details about the analytic continuation of $\kappa u_{2}$.
This normalization also implies that the local coordinate near the infinity cusp can be defined by
\begin{equation}
\tau_{\infty}:={\kappa \pi_{2}\over \pi_{1}}\,.
\end{equation}
It turns out that \cite{Berndt:1995}
\begin{equation}
\tau_{\infty}=\tau\,.
\end{equation}
Throughout this appendix we use both notations interchangeably. 

\paragraph{Local basis of solutions near $\alpha=1$}

The suitable basis at $[\tau]=[0]$ on the modular curve or equivalently equivalently $\beta:=1-\alpha=0$ on the base of the elliptic curve family is given by
\begin{eqnarray}
\pi_{2}&=&~_{2}F_{1}({1\over 3},{2\over 3},1,\beta)\,,\nonumber\\
\kappa \pi_{1}&=&{i\over \sqrt{3}}~_{2}F_{1}({1\over 3},{2\over 3},1,1-\beta)\,.
\end{eqnarray}
The local monodromy $M_{0}^{\mathrm{local}}$ near $\alpha=1$
is
\begin{equation}
M_{0}^{\mathrm{local}}=
 \left( \begin{array}{cc}
1& 1 \\
0 & 1
\end{array} \right)\,.
\end{equation}
The local coordinate near this cusp is given by
\begin{equation}
\tau_{0}:={\kappa \pi_{1}\over \pi_{2}}\,.
\end{equation}

\paragraph{Local basis of solutions near $\alpha=\infty$}

Near the elliptic point $[\tau]=[-\kappa \zeta_{3}]$ on the modular curve or equivalently the orbifold point $\alpha=\infty$ on the base of the elliptic curve family, the local basis is \cite{Erdelyi:1981}
\begin{eqnarray}
\pi_{3}&=&(-\alpha)^{-{1\over 3}}~_{2}F_{1}({1\over 3},{1\over 3},{2\over 3},\alpha^{-1})\,,\nonumber\\
\pi_{4}&=&(-\alpha)^{-{2\over 3}}~_{2}F_{1}({2\over 3},{2\over 3},{4\over 3},\alpha^{-1})\,.
\end{eqnarray}
The local monodromy in this basis is diagonal and of finite order:
\begin{equation}
M_{\mathrm{orb}}^{\mathrm{local}}=
 \left( \begin{array}{cc}
\zeta_{3} & 0\\
0& \zeta_{3}^{2}
\end{array} \right)\,.
\end{equation}

\paragraph{Connection matrices and analytic continuation}

The connection matrix $P$ between the two local bases in $(\pi_{2},\kappa \pi_{1})=(\pi_{1},\kappa \pi_{2}) P$
is easily seen to be
\begin{equation}
P=
 \left( \begin{array}{cc}
0 & {i\over \sqrt{3}} \\
{ \sqrt{3}\over i}  & 0
\end{array} \right)\,.
\end{equation}
It relates the local coordinates
$\tau_{\infty}$ and
$\tau_{0}$ in the following way
\begin{equation}
\tau_{0}=
 \left( \begin{array}{cc}
0 & {i\over \sqrt{3}} \\
{ \sqrt{3}\over i}  & 0
\end{array} \right) \cdot \tau_{\infty}=-{1\over 3\tau_{0}}
\end{equation}
This is nothing but the Fricke involution which was essential in \cite{Alim:2013eja} to study mirror symmetry of some non-compact CY 3-folds.
It induces, see \cite{Maier:2009}, the action on the Hauptmodul $\alpha\mapsto \beta$.
The moduli interpretation of the modular curve says that it is the moduli space of pairs $(E,C)$, where $C$ is a cyclic group of order three of the group $E[3]$ of $3$-torsion points of the elliptic curve $E$. Fricke involution is the map
\begin{equation}
W_{3}: (E,C)\mapsto  (E/C,E[3]/C)\,.
\end{equation}

The particular form of this connection matrix also answers the question
that why it is the Fricke involution instead of the S-transformation which naturally relates the local expansions of the modular forms. One could have chosen a different normalization of the periods so that $\tau_{\infty}$ and
$\tau_{0}$ are indeed related by $S$, but then either the new variable $q:=\exp (2\pi i \tau)$ would not be such that
$j={1/ q}+744+\cdots$ or the monodromy is not the expected transform $T$ which indicates the singular type $I_{1}$ of the degeneration of the elliptic curve at the singularity.

The monodromy at $\alpha=1$ in the original basis $(\pi_{1}, \kappa \pi_{2})$ is
\begin{equation}
M_{0}=PM_{0}^{\mathrm{local}}P^{-1}=
 \left( \begin{array}{cc}
 1& 0 \\
-3& 1
\end{array} \right)=-ST^{3}S\,.
\end{equation}

The connection matrix $Q$ between $(\pi_{3},\pi_{4})$ and
 $(\pi_{1},\pi_{2})$
satisfies, see \cite[Page 107]{Erdelyi:1981},
\begin{eqnarray}
\pi_{1}&=&{\Gamma(c)\Gamma(b-a)\over \Gamma(c-a)\Gamma(b)}\pi_{3}
+
{\Gamma(c)\Gamma(a-b)\over \Gamma(c-b)\Gamma(a)}\pi_{4}\,,\nonumber\\
\pi_{2}&=&{\Gamma(a+b+1-c)\Gamma(b-a)\over \Gamma(b+1-c)\Gamma(b)}e^{-\pi i a}\pi_{3}
+
{\Gamma(a+b+1-c)\Gamma(a-b)\over \Gamma(a+1-c)\Gamma(a)}e^{-\pi i b}\pi_{4}\,.
\end{eqnarray}
Hence 
\begin{equation}
Q^{-1}=
 \left( \begin{array}{cc}
 {\Gamma({1\over 3})\over \Gamma({2\over 3})^{2}}&  {\Gamma({1\over 3})\over \Gamma({2\over 3})^{2}} e^{-i\pi {1\over 3}}\\
 {\Gamma(-{1\over 3})\over \Gamma({1\over 3})^{2}}&   {\Gamma(-{1\over 3})\over \Gamma({1\over 3})^{2}}e^{-i\pi {2\over 3}}
\end{array} \right)\,.
\end{equation}
The monodromy near $[\tau]=[-\kappa \zeta_{3}]$ in the original basis is
\begin{equation}
M_{\mathrm{orb}}=QM_{\mathrm{orb}}^{\mathrm{local}}Q^{-1}=
 \left( \begin{array}{cc}
 -2& -1 \\
3& 1
\end{array} \right)\,.
\end{equation}
The consistency condition is checked to be satisfied:
\begin{equation}
M_{\mathrm{orb}}M_{0}M_{\infty}=\mathrm{Id}\,.
\end{equation}
For ease of notation, we denote
\begin{eqnarray}
\tilde{\pi}_{3}&=&\gamma_{+} \pi_{3}\,,\gamma_{+}= {\Gamma({1\over 3})\over \Gamma({2\over 3})^{2}}\,,\nonumber\\
\tilde{\pi}_{4}&=&\gamma_{-} \pi_{4}\,,\gamma_{-}= {\Gamma(-{1\over 3})\over \Gamma({1\over 3})^{2}}\,.
\end{eqnarray}
Similar to the local coordinates $\tau_{\infty},\tau_{0}$, one
defines the local coordinate near the orbifold point by
\begin{equation}\label{eqnormalizedperiodellipticpoint}
\tau_{\mathrm{orb}}:=K   {\tilde{\pi}_{4}\over\tilde{\pi}_{3}}\,.
\end{equation}
The overall constant $K$ is fixed by requiring that 
the connection matrix relating $\tau_{\infty}$ and $\tau_{\mathrm{orb}}$ is given by a fractional linear transform in $\mathrm{SL}_{2}(\mathbb{C})$:
\begin{equation}\label{eqlocaluniformizingellipticpoint}
\tau_{\infty}
={\kappa (e^{-\pi i a } \tilde{u}_{3}+e^{-\pi i b }  \tilde{u}_{4}) \over  ( \tilde{u}_{3}+\tilde{u}_{4})}
\,,\quad \tau_{\mathrm{orb}}={\tau_{\infty}-\kappa e^{-\pi i a}\over -{1\over K}\tau_{\infty}+{\kappa\over K} e^{-\pi i b}}\,.
\end{equation}
Therefore, we obtain
\begin{equation}
K=\kappa (e^{-\pi i b} -e^{-\pi i a} )\,.
\end{equation}
The expression for the connection matrix also tells that the elliptic fixed point given by $\tau_{\mathrm{orb}}=0$ has the 
$\tau$-value 
\begin{equation}
\tau_{*}=\kappa e^{-\pi i a}\,.
 \end{equation}
Note that the following identity holds due to the relation $a+b=1$
\begin{equation}
-K=\tau_{*}-\bar{\tau}_{*}\,.
\end{equation}

\end{appendices}

\providecommand{\bysame}{\leavevmode\hbox to3em{\hrulefill}\thinspace}

\bigskip{}

\noindent{\small Department of Mathematics,  Stanford University, Stanford, California 94305, USA}

\noindent{\small Email: \tt yfshen@stanford.edu}

\medskip{}
\noindent{\small Perimeter Institute for Theoretical Physics, 31 Caroline Street North, Waterloo, Ontario N2L 2Y5, Canada}

\noindent{\small Email: \tt jzhou@perimeterinstitute.ca}

\end{document}